\documentclass[10pt]{article}  
\usepackage{amsmath}
\usepackage{amssymb}
\usepackage{graphicx}
\usepackage{amsthm}       \newtheorem{thm}{Theorem}[section]
\newtheorem{prop}[thm]{Proposition}
\newtheorem{lem}[thm]{Lemma}
\newtheorem{cor}[thm]{Corollary}  \theoremstyle{definition}
\newtheorem{df}[thm]{Definition}   \theoremstyle{definition}

\newtheorem{prob}[thm]{Problem}
\newtheorem{rem}[thm]{Remark}                \theoremstyle{plain}
 \theoremstyle{definition}
\newtheorem{ex}[thm]{Example}   \def\CC{\Bbb{C}}
\def\RR{\Bbb{R}}  
\def\CCI{\hat{\CC}}        \def\NN{\Bbb{N}} 

\def\g{\gamma}
\def\G{\Gamma}
\def\GN{\Gamma ^{\NN }}
\def\B1{{\rm\kern.32em\vrule    width.12em       height1.4ex
depth-.05ex\kern-.28em 1}}
\def\pc{\pi _{\CCI }}

\def\GNCR{\GN \times \CCI \rightarrow \GN \times \CCI }
\begin{document}
\title
{Dynamics of 
postcritically bounded 
polynomial 
semigroups II: fiberwise dynamics and the Julia sets
\footnote{Date: November 29, 2013. Published in 
J. London Math. Soc. (2) 88 (2013) 294--318. 
2010 Mathematics Subject Classification. 
37F10, 37C60. This work was partially supported by JSPS Grant-in-Aid for
 Scientific Research (C) 21540216. 
Keywords: Polynomial semigroups, 
random complex dynamics, random iteration, 
skew product, 
Julia sets, fiberwise Julia sets.}}
\author{Hiroki Sumi\\ 
 Department of Mathematics,  
Graduate School of Science\\ 
Osaka University \\ 1-1, \ Machikaneyama,\ Toyonaka,\ Osaka,\ 560-0043,\ 
Japan\\ E-mail: sumi@math.sci.osaka-u.ac.jp\\ 
http://www.math.sci.osaka-u.ac.jp/$\sim $sumi/welcomeou-e.html}
\date{}
\maketitle 

\begin{abstract}
We investigate the dynamics of 
%polynomial semigroups 
%$G$
%(i.e., semigroups generated by polynomial maps on the Riemann sphere)  
semigroups generated by polynomial maps on the Riemann sphere 
such that the postcritical set in the complex plane is bounded. 
%The Julia set of such a semigroup may not be connected in general. 
%and 
%we see many examples. 
Moreover, we investigate the associated random dynamics of polynomials.
 Furthermore, we investigate the fiberwise dynamics of skew products 
 related to  
polynomial semigroups with bounded planar postcritical set. 
Using uniform 
fiberwise quasiconformal surgery 
on a fiber bundle, 
%Moreover, 
we show that 
if the Julia set of such a semigroup 
is disconnected, then 
%there exists 
there exist families of uncountably many mutually disjoint quasicircles 
with uniform dilatation which are parameterized by the Cantor set, densely 
inside the Julia set of the semigroup. 
Moreover, we give a sufficient condition for a fiberwise Julia set $J_{\gamma }$ 
to satisfy that $J_{\gamma }$ is a Jordan curve but not a quasicircle, 
the unbounded component of $\CCI \setminus J_{\gamma }$ is a John domain and 
the bounded component of $\CC \setminus J_{\gamma }$ is not a John domain.   
We show that 
%in 
under certain conditions, 
%almost surely a random Julia set is a Jordan curve but not a quasicircle.
a random Julia set is almost surely a Jordan curve, but not a quasicircle.
Many new phenomena of polynomial semigroups and random dynamics of 
polynomials that do not occur in the usual dynamics of polynomials  
are found and systematically investigated.
\end{abstract}

\section{Introduction}
 The theory of complex dynamical systems, which has 
 its origin in the important 
 work of Fatou and Julia in the 1910s, 
 has been investigated by many people and discussed in depth.  
In particular, since D. Sullivan showed the famous 
``no wandering domain theorem'' using 
%the 
Teichm\"{u}ller theory 
in the 1980s, 
this subject has 
%been attracting 
attracted 
many researchers 
%in 
from a
wide area. 
%For general 
For a general 
reference on complex dynamical systems, 
see Milnor's textbook \cite{M}.   
 
 There are several 
% directions 
areas
 in which we deal with 
  generalized notions of 
classical iteration theory of rational functions.   
One of them is the theory of 
dynamics of rational semigroups 
  (semigroups generated by holomorphic maps on the 
  Riemann sphere $\CCI $), and another one is 
 the theory of   
 random dynamics of holomorphic maps on the Riemann sphere. 

%We will explain the new directions mentioned above.
In this paper, we will discuss 
%the subjects mentioned above. 
these subjects.
 A {\bf rational semigroup} is a semigroup 
generated by a family of non-constant rational maps on 
$\CCI $, where $\CCI $ denotes the Riemann sphere,
 with the semigroup operation being  
functional composition (\cite{HM1}). A 
{\bf polynomial semigroup} is a 
semigroup generated by a family of non-constant 
polynomial maps.
Research on the dynamics of
rational semigroups was initiated by
A. Hinkkanen and G. J. Martin (\cite{HM1}),
who were interested in the role of the
dynamics of polynomial semigroups while studying
various one-complex-dimensional
moduli spaces for discrete groups,
and
by F. Ren and Z. Gong (\cite{GR}) and others, 
 who studied 
such semigroups from the perspective of random dynamical systems.
Moreover, the research 
%of 
on 
rational semigroups is related to 
that 
%of 
on 
``iterated function systems" in 
%the fractal geometry. 
fractal geometry.  
%Actually, 
In fact, 
the Julia set of a rational semigroup generated by a 
compact family has 
%the 
`` backward self-similarity" 
(cf. \cite{S3,S1}). 
\cite{St} is a very nice (and short) article for an introduction to the dynamics of rational semigroups. 
For 
%any 
other 
%researches 
research 
on rational semigroups, see 
\cite{SY, SSS, 
SS, SU1, SU2}, and \cite{S5}--\cite{Scpsb}. 

 Research 
% of 
on the  
dynamics of rational semigroups is also directly related to 
that 
%of 
on the 
random dynamics of holomorphic maps. 
The first 
%research 
study 
in this 
direction was 
%given 
by Fornaess and Sibony (\cite{FS}), and 
%many researches have been done. 
much research has followed. 
(See \cite{Br, Bu1, Bu2, 
BBR,GQL,S9, SdpbpIII,Splms,Ssugexp, Scpsb,S8}.)   

 We remark that complex dynamical systems 
 can be used to describe some mathematical models. For 
 example, the behavior of the population 
 of a certain species can be described as the 
 dynamical system of a polynomial 
 $f(z)= az(1-z)$ 
 such that $f$ preserves the unit interval and 
 the postcritical set in the plane is bounded 
 (cf. \cite{D}). From this point of view, 
 it is very important to consider the random 
 dynamics of such polynomials (see also Example~\ref{realpcbex}). 
 The results of this paper might have applications to mathematical models. 
For the random dynamics of polynomials on the unit interval, 
see \cite{Steins}. 
 
 We shall give some definitions 
% on 
for the 
dynamics of rational semigroups: 
\begin{df}[\cite{HM1,GR}] 
Let $G$ be a rational semigroup. We set
\[ F(G) = \{ z\in \CCI \mid G \mbox{ is normal in a neighborhood of  $z$} \} ,
\ J(G)  = \CCI \setminus  F(G) .\] \(  F(G)\) is  called the
{\bf Fatou set}  of  $G$ and \( J(G)\)  is  called the {\bf 
Julia set} of $G$. 
%The backward orbit $G^{-1}(z)$ of $z$ and 
%{\bf the 
%set of exceptional points} $E(G)$ are defined\
% by:
% $ G^{-1}(z) =\cup _{g\in G}  g^{-1}(z) $ and 
%$ E(G)= \{z\in \CCI \mid \sharp G^{-1}(z) <\infty  \} .$
%For any subset $A$ of $\CCI ,\ $ we set 
%$G^{-1}(A)=\cup _{g\in G}g^{-1}(A).$  
We %denote by 
let 
$\langle h_{1},h_{2},\ldots \rangle $ 
denote 
the 
rational semigroup generated by the family $\{ h_{i}\} .$
The Julia set of the semigroup generated by 
a single map $g$ is denoted by 
$J(g).$ 
\end{df}

\begin{df}
%\begin{enumerate}
%\item 
For each rational map $g:\CCI \rightarrow \CCI $, 
we set 
$CV(g):= \{ \mbox{all critical values of }
g: \CCI \rightarrow \CCI \} .$ 
Moreover, for each polynomial map $g:\CCI \rightarrow \CCI $, 
we set $CV^{\ast }(g):= CV(g)\setminus \{ \infty \} .$ 
%\item 
For a rational semigroup $G$, 
we set 
$$ P(G):=
%\overline{\bigcup _{g\in G\cup }\{ \mbox{all critical values of }g:\CCI 
%\rightarrow \CCI \} } \ (\subset \CCI )
\overline{\bigcup _{g\in G} CV(g)} \ (\subset \CCI ). 
$$ 
This is called the {\bf postcritical set} of $G.$
Furthermore, for a polynomial semigroup $G$,\ we set 
$P^{\ast }(G):= P(G)\setminus \{ \infty \} .$ This is 
called the {\bf planar postcritical set}
(or {\bf finite postcritical set}) 
 of $G.$
We say that a polynomial semigroup $G$ is 
{\bf postcritically bounded} if 
$P^{\ast }(G)$ is bounded in $\CC .$ 
%\end{enumerate}
\end{df}
\begin{rem}
\label{pcbrem}
Let $G$ be a rational semigroup 
generated by a family $\Lambda $ of rational maps. 
Then, we have that 
$P(G)=\overline{\cup _{g\in G\cup \{ Id\} }\ g(\cup _{h\in \Lambda }CV(h))}$, 
where Id denotes the identity map on $\CCI $,  
and that $g(P(G))\subset P(G)$ for each $g\in G.$  
Using this formula, one can understand how the set 
$P(G)$ (resp. $P^{\ast }(G)$) spreads in $\CCI $ (resp. $\CC $). 
In fact, in Section~\ref{Const}, using the above formula, 
we present a way to construct examples of postcritically bounded 
polynomial semigroups (with some additional properties). Moreover, 
from the above formula, one may, in the finitely generated case, 
use a computer to see if a polynomial semigroup $G$ is postcritically bounded much in the same way 
as one verifies the boundedness of the critical orbit for the maps $f_{c}(z)=z^{2}+c.$   
\end{rem}
\begin{ex}
\label{realpcbex}
Let 
$\Lambda := \{ h(z)=cz^{a}(1-z)^{b}\mid 
a,b\in \NN  ,\ c>0,\  
c(\frac{a}{a+b})^{a}(\frac{b}{a+b})^{b}$ $\leq 1\} $ 
and let $G$ be the polynomial semigroup generated by 
$\Lambda .$ 
Since for each $h\in \Lambda $, 
$h([0,1])\subset [0,1]$ and 
$CV^{\ast }(h)\subset [0,1]$, 
it follows that each subsemigroup $H$ of $G$ is postcritically 
bounded. 
\end{ex}
\begin{rem}
\label{pcbound}
It is well-known that for a polynomial $g$ with 
$\deg (g)\geq 2$, 
$P^{\ast }(\langle g\rangle )$ is bounded in $\CC $ if and only if 
$J(g)$ is connected (\cite[Theorem 9.5]{M}).
\end{rem}
As mentioned in Remark~\ref{pcbound}, 
 the planar postcritical set is one 
% of the important 
piece of important information 
% , 
regarding the dynamics of polynomials. 
% Actually, in 
%Concerning 
%the theory of iteration of quadratic polynomials, 
%we have been investigating the famous ``Mandelbrot set''.
    
When investigating the dynamics of polynomial semigroups, 
it is natural for us to 
discuss the relationship between 
  the planar postcritical set and the 
%  figure of the 
Julia set.
%  , in order to investigate the dynamics of polynomial 
% semigroups.    
The first question in this 
%direction 
regard 
is: 
% whether a similar statement for polynomial semigroups 
% is true or not: 
%\begin{ques}
``Let $G$ be a polynomial semigroup such that each 
element $g\in G$ is of degree at least two.
%If $P^{\ast }(G)$ is bounded 
%in $\CC ,\ $ then is $J(G)$ connected? 
Is $J(G)$ necessarily connected when $P^{\ast }(G)$ is 
bounded in $\CC $?'' 
%\end{ques}
The answer is {\bf NO.} In fact, in \cite{SY, SdpbpI, SdpbpIII, SS, Splms, Ssugexp}, 
we find many examples of postcritically bounded polynomial semigroups $G$ 
with disconnected Julia set such that for each $g\in G$, $\deg (g)\geq 2.$ 
%\begin{ex}[\cite{SY}]
%Let $G=\langle z^{3}, \frac{z^{2}}{4}\rangle .$ 
%Then $P^{\ast }(G) =\{ 0\} $ 
%(which is bounded in $\CC $)
%and $J(G)$ is disconnected ($J(G)$ is a Cantor set 
%of round circles). Furthermore,\ 
%by 
%according to 
%\cite[Theorem 2.4.1]{S5},  
%it can be shown that 
%small 
%a small					 
%perturbation $H$ of $G$ 
% still satisfies that 
% $P^{\ast }(H) $ is 
% bounded in $\CC $  and that $J(H)$ is disconnected. 
% ($J(H)$ is a 
%Cantor set of quasi-circles with uniform dilatation.)
%\end{ex}
Thus, it is natural to ask the following 
%problems. 
questions. 
\begin{prob}
(1) What properties does $J(G)$ have 
%Investigate what happens 
if $P^{\ast }(G) $ is bounded in $\CC $ 
and $J(G)$ is disconnected?    
%\end{ques}
%\begin{prob}
(2) Can we classify postcritically bounded polynomial semigroups? 
%\end{prob}
\end{prob} 
 Applying the results in \cite{SdpbpI,SdpbpIII},  
 we investigate the dynamics of 
every sequence, or fiberwise dynamics of the skew product 
associated with the generator system (cf. Section \ref{fibsec}). 
Moreover, we investigate 
%random dynamics 
the random dynamics of polynomials 
%in 
acting on 
the Riemann 
sphere. 
Let us consider a polynomial semigroup $G$ generated by a compact 
family $\G $ of polynomials. For each sequence 
$\g =(\g _{1},\g _{2},\g _{3},\ldots )\in \GN $, 
we examine the dynamics along the sequence $\g $, 
that is, the dynamics of the family of maps 
$\{ \g _{n}\circ \cdots \circ \g _{1}\} _{n=1}^{\infty }$. 
We note that this corresponds to the fiberwise dynamics of the skew 
product (see Section \ref{fibsec}) associated with the generator 
system $\G .$ 
We show that 
if $G$ is postcritically bounded, $J(G)$ is disconnected, 
and $G$ is generated by a compact family $\G $ of 
%polynomials, 
polynomials,  
then, for almost every sequence $\g \in \GN $, 
there exists exactly one bounded component $U_{\g }$ 
of 
%Fatou set 
the Fatou set 
of $\g $, the Julia set of $\g $ has Lebesgue measure zero, 
there exists no non-constant limit function in $U_{\g }$ for the 
sequence $\g $, 
and for any point $z\in U_{\g }$  
the orbit of $z$ along $\g $ 
tends to the interior of the smallest filled-in Julia set 
$\hat{K}(G)$ (see Definition~\ref{d:smfj}) of $G$ (cf. Theorem~\ref{mainth3},  
Corollary~\ref{rancor1}). 
 Moreover, using uniform fiberwise quasiconformal surgery 
 (\cite{SdpbpIII}), 
 we find 
% sub skew products 
sub-skew products $\overline{f}$ 
 such that $\overline{f}$ is hyperbolic (see Definition~\ref{d:hypskew}) and 
 such that every fiberwise Julia 
% sets 
set 
of $\overline{f}$ is a $K$-quasicircle, 
 where $K$ is a constant not depending 
 on the fibers (cf. Theorem~\ref{mainth3}, statement \ref{mainth3-3}). 
%Using 
Reusing 
the uniform fiberwise quasiconformal surgery,
% again, 
we show that if $G$ is a postcritically bounded polynomial semigroup 
with disconnected Julia set, then for any non-empty 
open subset $V$ 
of $J(G)$, 
there exists a 
$2$-generator subsemigroup $H$ of $G$ such that 
$J(H)$ is the disjoint union of a ``Cantor family of quasicircles" 
(a family of quasicircles parameterized by a Cantor set) with uniform 
distortion, and such that $J(H)\cap V\neq \emptyset $ 
(cf. Theorem~\ref{cantorqc}). 
Note that the uniform fiberwise quasiconformal surgery 
is based on solving uncountably many Beltrami  
equations.
% (a kind of partial differential equations). 

 We also investigate (semi-)hyperbolic (see Definition~\ref{d:hypsh}), postcritically bounded, 
 polynomial semigroups generated by a compact family $\G $ of 
 polynomials.   
Let $G$ be such a semigroup with disconnected Julia set, 
and suppose that there exists an element $g\in G$ such that 
$J(g)$ is not a Jordan curve. Then, 
we give a (concrete) sufficient condition for a sequence $\g \in \GN $ to 
%satisfy that 
give rise to 
the following situation ($\ast $): 
the Julia set of $\g $ is a 
Jordan curve but not a quasicircle, the basin of infinity $A_{\g }$ is a 
John domain, and the bounded component $U_{\g }$ of 
the Fatou set is not a John domain (cf. Theorem~\ref{mainthjbnq}, Corollary~\ref{c:rancor2sh}).
From this result, we show that 
for almost 
%sure 
every sequence $\g \in \GN $, situation $(\ast )$ holds. 
In fact, in this paper, 
under the above assumption, we find a set ${\cal A}$ of $\g $ with $(\ast )$ which is much larger than a set 
${\cal U}$ of $\g $ with $(\ast )$ given in \cite{SdpbpIII}. 
Moreover, we classify hyperbolic two-generator postcritically bounded polynomial semigroups $G$ with disconnected Julia set 
and we also completely classify the fiberwise Julia sets $J_{\g }$ in terms of the information of $\g $ (Theorem~\ref{t:tghyp}).     
Note that situation $(\ast )$ cannot hold in the usual iteration dynamics of a single polynomial map $g$ with $\deg (g)\geq 2$ 
(Remark~\ref{r:jbnq}). 

The key to 
%investigate 
investigating the 
dynamics of 
postcritically bounded polynomial semigroups is 
%,  
the density of repelling fixed points in the Julia set (\cite{HM1,GR}), 
which 
%is 
can be shown by 
an application of 
%Ahlfors's
the Ahlfors five island theorem, and the lower semi-continuity of 
$\g \mapsto J_{\g }$ (\cite{J}), which is a consequence of potential theory.  
Moreover, one of the keys to 
%investigate 
investigating the 
fiberwise dynamics of 
skew products is, the observation of non-constant limit functions
(cf. Lemma~\ref{nclimlem} and \cite{S1}). 
 The key to 
% investigate 
investigating the 
dynamics of semi-hyperbolic polynomial semigroups is, 
the continuity of the map $\g \mapsto J_{\g }$ 
%(this is not trivial)  
(this is highly nontrivial; see \cite{S1}) 
%at all) 
and the 
Johnness of the basin $A_{\g }$ of infinity (cf. \cite{S4}). 
Note that the continuity of the map $\g \mapsto J_{\g }$ 
%above 
does not hold in general, if we do not assume semi-hyperbolicity. 
 Moreover, one of the 
% originality 
original aspects 
of this paper 
%is,  
is the idea of 
 ``combining both 
% theory 
the theory of rational semigroups and 
 that of random complex dynamics". It is quite natural to 
investigate both fields simultaneously. However, 
%there has been no research done investigating 
no study (except the works of the author of this paper)
%has 
%investigated 
%both fields. 
thus far has done so.

 Furthermore, in Section~\ref{Const} and \cite{SdpbpI,SdpbpIII}, we 
%give 
provide 
a way of constructing examples of 
 postcritically bounded polynomial semigroups 
 with 
% disconnected Julia 
% sets 
%set
some additional properties (disconnectedness of the Julia set, 
semi-hyperbolicity, hyperbolicity, etc.) 
(cf. Proposition~\ref{Constprop}, \cite{SdpbpI,SdpbpIII}). 
For example, by Proposition~\ref{Constprop}, 
there exists a $2$-generator 
postcritically bounded polynomial semigroup $G=\langle h_{1},h_{2}\rangle $ 
with disconnected 
Julia set such that $h_{1}$ has a Siegel disk.
        
As we see in Example~\ref{realpcbex},  Section~\ref{Const}, and \cite{SdpbpI, SdpbpIII},  
it is not difficult to construct many examples for which   
we can verify the hypothesis ``postcritically 
bounded'', 
so the class of postcritically bounded polynomial semigroups is 
very wide.  

 Throughout the paper, we will see many new phenomena in polynomial 
 semigroups or random dynamics of polynomials that do not occur in 
 the usual dynamics of polynomials. Moreover, these new phenomena are 
 systematically investigated.
 
 In Section~\ref{Main}, we present the main results 
 of this paper. We give some tools in Section~\ref{Tools}. 
 The proofs of the main results are given in Section~\ref{Proofs}.

There are many applications of the results of postcritically 
bounded polynomial semigroups in many directions. 
 In the sequel \cite{Splms, S9, Scpsb, S8}, we 
 investigate 
% random 
the Markov process on $\CCI $ associated with 
the random dynamics of polynomials and 
% associated ``singular functions in the complex plane", 
% using some 
% results in this paper. 
%of the results from this paper.
we  consider the probability $T_{\infty }(z)$ 
of tending to $\infty \in \CCI $ 
starting with the initial value $z\in \CCI .$ 
Applying many results of \cite{SdpbpI},  
it will be shown in \cite{S8} that 
if the associated polynomial semigroup $G$ 
is postcritically bounded and the Julia set is 
disconnected, then 
the chaos of the averaged system disappears due to the cooperation of generators (cooperation principle), 
and 
the function $T_{\infty }$ defined on $\CCI $ 
has many interesting properties which are 
similar to those of the devil's staircase (the Cantor function). 
Such  ``singular functions on the 
complex plane'' appear very naturally in 
random dynamics of polynomials and the 
results of this paper (for example, 
the results on the space of all connected 
components of a Julia set) are the keys to 
investigating them. 
(The above results have been announced in \cite{Splms, S9, S10, Ssugexp}.) 

% Moreover, as illustrated before, 
% it is very important for us to recall that 
% the complex dynamics can be applied to describe some mathematical models. 
% For example, the behavior of the population of a 
% certain species can be described as the dynamical systems
%  of a polynomial $h$ such that $h$ preserves the unit interval  
% and the postcritical set in the plane is bounded. 
% When one considers such a model, 
% it is very natural to consider the random dynamics of 
% polynomial with bounded postcritical set in the plane 
% (see Example~\ref{realpcbex}).    

In \cite{SdpbpI}, we find many fundamental and useful results on the connected components of Julia sets of postcritically bounded polynomial semigroups.  
In \cite{SdpbpIII}, we classify (semi-)hyperbolic, postcritically bounded, compactly generated polynomial semigroups. 
 In the sequel \cite{SS}, we give some 
% other 
further results on postcritically 
 bounded polynomial semigroups, by using many results in \cite{SdpbpI,SdpbpIII}, and this paper. 
 Moreover, in the sequel \cite{S7}, 
 we define a new kind of cohomology theory, in order to 
 investigate the action of finitely generated semigroups (iterated function systems), and 
% will 
we apply it to the study of 
% dynamics 
the dynamics of postcritically bounded polynomial semigroups.  

\section{Preliminaries}  
\label{Preliminaries}
%In this section, we present the main results of this paper. 
%For any polynomial $g$, we set 
%  $K(g):=\{ z\in \CC \mid \cup _{n\in \NN }\{g^{n}(z)\}
%  \mbox{ is bounded in }\CC \} .$ 
In this section we give some basic notations and definitions, and 
we present some results in \cite{SdpbpI,SdpbpIII}, which we need to state
 the main results of this paper. 
%\subsection{Space of connected components of a Julia set, surrounding order}
%\label{concompsec}
%We present some results 
%on 
%concerning the connected components of the 
%Julia set of a postcritically bounded polynomial semigroup.
%with bounded postcritical set in the 
%plane. 
\begin{df}
We set 
Rat : $=\{ h:\CCI \rightarrow \CCI \mid 
h \mbox { is a non-constant rational map}\} $
endowed with distance $\eta $ defined as 
$\eta (f,g):= \sup _{z\in \CCI }d(f(z),g(z))$, 
where $d$ denotes the spherical distance on $\CCI .$ 
%the topology induced by uniform convergence on $\CCI $ 
%with respect to the spherical distance.   
We set 
Poly :$=\{ h:\CCI \rightarrow \CCI 
\mid h \mbox{ is a non-constant polynomial}\} $ endowed with 
the relative topology from Rat.   
Moreover, we set 
Poly$_{\deg \geq 2}
:= \{ g\in \mbox{Poly}\mid \deg (g)\geq 2\} $ 
endowed with the relative topology from 
Rat.  
\end{df}
\begin{rem}
Let $d\geq 1$,  $\{ p_{n}\} _{n\in \NN }$ be a 
sequence of polynomials of degree $d$, 
and $p$ be a polynomial.  
Then $p_{n}\rightarrow p$ in Poly if and only if $p$ is of degree $d$ and  
the coefficients of $p_{n}$ converge appropriately.  
\end{rem}
\begin{df} 
Let ${\cal G} $ be the set of all postcritically bounded polynomial semigroups 
$G$ such that 
 each element of $G$ is of degree 
at least two.    
Furthermore, we set 
${\cal G}_{con}=
\{ G\in {\cal G}\mid 
J(G)\mbox{ is connected}\} $ and 
${\cal G}_{dis}=
\{ G\in {\cal G}\mid 
J(G)\mbox{ is disconnected}\}.$ 
\end{df}  
\begin{df} 
For a polynomial semigroup $G$,\ 
we denote by 
${\cal J}={\cal J}_{G}$ the set of all 
connected components $J$ 
of $J(G)$ such that $J\subset \CC .$   
Moreover, we denote by 
$\hat{{\cal J}}=\hat{{\cal J}}_{G}$ the set of all connected components 
of $J(G).$ 
\end{df} 
\begin{rem}
\label{hatjcptrem}
If a polynomial semigroup $G$ is generated by a compact set 
in Poly$_{\deg \geq 2}$, then 
$\infty \in F(G)$ and thus ${\cal J}=\hat{{\cal J}}.$ 
\end{rem}
\begin{df}[\cite{SdpbpI}] 
For any connected sets $K_{1}$ and 
$K_{2}$ in $\CC ,\ $  ``$K_{1}\leq K_{2}$'' indicates that 
$K_{1}=K_{2}$, or $K_{1}$ is included in 
a bounded component of $\CC \setminus K_{2}.$ Furthermore, 
``$K_{1}<K_{2}$'' indicates $K_{1}\leq K_{2}$ 
and $K_{1}\neq K_{2}.$ Note that 
``$ \leq $'' is a partial order in 
the space of all non-empty compact connected 
sets in $\CC .$ This ``$\leq $" is called 
the {\bf surrounding order.} 
\end{df}
\begin{df}[\cite{SdpbpI}] 
\label{d:smfj}
 For a polynomial semigroup $G$,\ we set 
$$ \hat{K}(G):=\{ z\in \CC 
\mid \bigcup _{g\in G}\{ g(z)\} \mbox{ is bounded in }\CC \} $$ 
and call $\hat{K}(G)$ the {\bf smallest filled-in Julia set} of 
$G.$ 
For a polynomial $g$, we set $K(g):= \hat{K}(\langle g\rangle ).$ 
%\end{df}
For a set $A\subset \CCI $, we denote by int$(A)$ the set of 
all  
interior points of $A.$ 
For a polynomial semigroup $G$ with 
$\infty \in F(G)$, we denote by 
$F_{\infty }(G)$ the connected component of $F(G)$ containing 
$\infty .$ Moreover, for a polynomial $g$ with 
$\deg (g)\geq 2$, we set $F_{\infty }(g):= 
F_{\infty }(\langle g\rangle ).$ 
\end{df}

The following three results in \cite{SdpbpI} are needed to state the main result in this paper. 

\begin{thm}[\cite{SdpbpI}] 
\label{mainth1}
Let $G\in {\cal G}$ (possibly generated by a non-compact 
family). 
Then we have all of the following. 
\begin{enumerate}
\item \label{mainth1-1}
We have $({\cal J},\ \leq )$ is totally ordered. 
\item \label{mainth1-2}
Each connected component of 
$F(G)$ is either simply or doubly connected. 
\item \label{mainth1-3}
For any $g\in G$ and any connected component 
$J$ of $J(G)$,\ we have that  
$g^{-1}(J)$ is connected. 
Let $g^{\ast }(J)$ be the connected component of 
$J(G)$ containing $g^{-1}(J).$ 
If $J\in {\cal J}$, then 
$g^{\ast }(J)\in {\cal J}.$    
If $J_{1},J_{2}\in {\cal J} $ and $J_{1}\leq J_{2},\ $ then 
$g^{-1}(J_{1})\leq g^{-1}(J_{2})$ 
and $g^{\ast }(J_{1})\leq g^{\ast }(J_{2}).$
\end{enumerate} 
\end{thm}
\begin{thm}[\cite{SdpbpI}] 
\label{mainth2} 
Let $G\in {\cal G}_{dis}$ (possibly generated by a non-compact 
family). 
Then we have all of the following. 
\begin{enumerate}
\item \label{mainth2-1}
We have $\infty \in F(G)$. Thus ${\cal J}=\hat{{\cal J}}.$ 
\item \label{mainth2-2}
%If $\infty \in F(G),\ $ 
The component  
$F_{\infty }(G)$ of $F(G)$ containing $\infty $ 
is simply connected. 
Furthermore,\ 
the element $J_{\max }=J_{\max}(G)\in {\cal J}$  
containing $\partial F_{\infty }(G)$ 
is the unique element of ${\cal J}$ satisfying that 
$J\leq J_{\max }$ for each 
$J\in {\cal J}.$  
\item 
\label{mainth2-3}
%Whatever the type of $G$ is, 
There exists a unique element 
$J_{\min }=J_{\min }(G)\in {\cal J}$ such that 
$J_{\min }\leq J$ for 
  each element $J\in {\cal J}. $
\item 
\label{mainth2-4}
%$\mbox{ int }\hat{K}(G)\neq 
%\emptyset $ if and only if 
%$G$ is of type (I) or (II)(a)(i)or (II)(b)(i). 
We have that  $\mbox{{\em int}}(\hat{K}(G))\neq 
\emptyset .$ 
\end{enumerate}
\end{thm} 
For the figures of the Julia sets of semigroups $G\in {\cal G}_{dis}$, 
see Figure~\ref{fig:dcgraph}. 
%\subsection{Properties of ${\cal J}$}
%\label{Properties}
\begin{prop}[\cite{SdpbpI}] 
\label{bminprop}
Let $G$ be a 
polynomial 
semigroup generated by a compact subset 
$\G $ of {\em Poly}$_{\deg \geq 2}.$ Suppose that $G\in {\cal G}_{dis}.$ Then,    
there exists an element 
$h_{1}\in \G $ with 
 $J(h_{1})\subset J_{\max } $ and 
there exists an element 
$h_{2}\in \G $ with 
$J(h_{2})\subset J_{\min }.$ 
\end{prop}
%\begin{figure}[htbp]
%\caption{The Julia set of a $3$-generator polynomial semigroup $G\in {\cal G}_{dis}$ with 
%$\sharp (\hat{{\cal J}}_{G})=\aleph _{0}.$}    
%\ \ \ \ \ \ \ \ \ \ \ \ \ \ \ \ \ \ \ \ \ \ \ \ \ \ \ \ \ \ \ \ 
%\includegraphics[width=4.9cm,width=4.9cm]{3mapcountjulia2.eps}
%\label{fig:3mapcountjulia2}
%\end{figure}
\section{Main results}
\label{Main}
In this section, we present the main results of this paper. 
The proofs of the results are given in Section~\ref{Proofs}. 
\subsection{Fiberwise dynamics and Julia sets}
\label{fibsec}
We present some results on the fiberwise 
dynamics of the skew product related to a postcritically bounded 
polynomial semigroup with disconnected Julia set.
 In particular, using the uniform 
 fiberwise quasiconformal surgery on a 
 fiber bundle, 
 we show the existence 
 of  families of quasicircles with uniform distortion 
parameterized 
 by the Cantor set densely inside  
 the Julia set of such a semigroup. The proofs 
 are given in Section~\ref{pffibsec}. 
\begin{df}[\cite{S1,S4}]
\label{d:sp}
\ 
\begin{enumerate}
\item  Let $X$ be a compact metric space, 
$g:X\rightarrow X$ a continuous 
map, and $f:X\times \CCI \rightarrow 
X\times \CCI $ a continuous map. 
We say that $f$ is a rational skew 
product (or fibered rational map on the 
trivial bundle $X\times \CCI $) 
over $g:X\rightarrow X$, if 
$\pi \circ f=g\circ \pi $ where 
$\pi :X\times \CCI \rightarrow X$ denotes 
the canonical projection,\ and 
if, for each $x\in X$, the 
restriction 
$f_{x}:= f|_{\pi ^{-1}(\{ x\} )}:\pi ^{-1}(\{ x\}) \rightarrow 
\pi ^{-1}(\{ g(x)\} )$ of $f$ is a 
non-constant rational map,\ 
under the canonical identification 
$\pi ^{-1}(\{ x'\} )\cong \CCI $ for each 
$x'\in X.$ Let $d(x)=\deg (f_{x})$, for each 
$x\in X.$
Let $f_{x,n}$ be 
%a 
the rational map 
defined by: $f_{x,n}(y)=\pi _{\CCI }(f^{n}(x,y))$, 
for each $n\in \NN ,x\in X$ and $y\in \CCI $,   
%Moreover, let $q_{x}:= q_{x}^{(1)}.$ 
where    
$\pi _{\CCI }:X\times \CCI 
\rightarrow  \CCI $ is the projection map.

  Moreover, if $f_{x,1}$ is a polynomial for each $x\in X$, 
 then we say that $f:X\times \CCI \rightarrow X\times \CCI $ is a 
 polynomial skew product over $g:X\rightarrow X.$  
\item  
%Let $G=\langle h_{1},\cdots ,h_{m}\rangle $ be a 
%finitely generatd rational semigroup. For a fixed generator 
%system $\{ h_{1},\cdots ,h_{m}\} $,\ 
Let $\G $ be a compact subset of Rat. 
We set  $\GN := \{ \g =(\g _{1}, \g_{2},\ldots )\mid \forall j,\g _{j}\in \G\} $ endowed with the product topology. This is a compact metric space.
Let  
%$\Sigma _{m}=\{ 1,\cdots ,m\} ^{\NN },\ 
$\sigma :\GN \rightarrow \GN $ be the shift map, which is defined by  
$\sigma (\g _{1},\g _{2},\ldots ):=(\g _{2},\g _{3},\ldots )
.$ Moreover,\ 
we define a map
$f:\GN \times \CCI \rightarrow 
\GN \times \CCI $ by:
$(\g ,y) \mapsto (\sigma (\g ),\g _{1}(y)),\ $
where $\g =(\g_{1},\g_{2},\ldots ).$
This is called 
{\bf the skew product associated with 
the family $\G $ of rational maps.} 
Note that $f_{\g ,n}(y)=\g _{n}\circ \cdots \circ \g _{1}(y).$ 
\end{enumerate}
\end{df}
\noindent \begin{rem}\ 
%\begin{itemize}
%\item[(1)] 
Regarding item 1 of Definition~\ref{d:sp}, 
the map $f^{n}|_{\pi ^{-1}(\{ x\} )}:\pi ^{-1}(\{ x\} )\rightarrow \pi ^{-1}(\{ g^{n}(x)\} )$ 
is equal to the rational map $f_{g^{n-1}(x)}\circ \cdots \circ f_{x}$ under the canonical 
identification $\pi ^{-1}(\{ x'\} )\cong \CCI $ for each $x'\in X.$ 
Thus, if we consider the dynamics of $f$, then we can investigate the dynamics of all sequences 
generated by the family $\{ f_{x}\} _{x\in X}$ and the map $g.$ 
%\item[(2)] 
%\end{itemize}
\end{rem}
\begin{rem}
Let $f:X\times \CCI \rightarrow X\times \CCI  $ 
be a rational skew product over 
$g:X\rightarrow X.$ Then the function 
$x\mapsto d(x)$ is continuous in $X.$ 
For, since $f$ is continuous, the map $x\mapsto f_{x}\in \mbox{Rat}$ is continuous. 
Moreover, the function $g\in \mbox{Rat}\mapsto \deg (g)\in \RR $ is continuous (\cite[Theorem 2.8.2]{Be}). 
Thus, $x\mapsto d(x)$ is continuous.    
\end{rem}
\begin{df}[\cite{S1, S4}]
Let $f:X\times \CCI 
\rightarrow X\times \CCI $ be a 
rational skew product over 
$g:X\rightarrow X.$ Then, 
%we use the following notation. 
%\begin{enumerate}
%\item 
for each $x\in X$ and $n\in \NN $, we set 
$f_{x}^{n}:=
f^{n}|_{\pi ^{-1}(\{ x\} )}:\pi ^{-1}(\{ x\} )\rightarrow 
\pi ^{-1}(\{ g ^{n}(x)\} )\subset X\times \CCI .$
%\item 
For each $x\in X$,  
we denote by $F_{x}(f)$ 
the set of points 
$y\in \CCI $ which has a neighborhood $U$ 
in $\CCI $ such that 
$\{ f_{x,n}:U\rightarrow 
\CCI \} _{n\in \NN }$
is normal. Moreover, we set 
$F^{x}(f):= \{ x\} \times F_{x}(f)\ (\subset X\times \CCI ).$  
%\item 
%For each $x\in X$, 
We set 
$J_{x}(f):=\CCI \setminus 
F_{x}(f).$ Moreover, we set 
$J^{x}(f):= \{ x\} \times J_{x}(f)$ $ (\subset X\times \CCI ).$ 
These sets $J^{x}(f)$ and $J_{x}(f)$ are called the 
{\bf fiberwise Julia sets}.
%\item 
Moreover, we set 
$\tilde{J}(f):=
\overline {\bigcup _{x\in X}J^{x}
(f)}$, where the closure is taken in the product space $X\times \CCI .$
%\item 
For each $x\in X$, we set 
$\hat{J}^{x}(f):=\pi ^{-1}(\{ x\} )\cap \tilde{J}(f).$ 
Moreover, we set $\hat{J}_{x}(f):= 
\pc (\hat{J}^{x}(f)).$ 
%\item 
We set $\tilde{F}(f):=(X\times \CCI)\setminus 
\tilde{J}(f).$
%\end{enumerate} 
\end{df}
\begin{rem}[\cite{S1, S4}]\ 
\begin{itemize}
\item[(1)]
We have 
$f_{x}^{-1}(J_{g(x)})=J_{x}$, 
$f_{x}(J_{x})=J_{g(x)}$, $f(\tilde{J}(f))\subset \tilde{J}(f)$, 
$\hat{J}^{x}(f)\supset J^{x}(f)$ and 
$\hat{J}_{x}(f)\supset J_{x}(f).$ 
However, for the last one, 
strict containment can occur. 
For example, let $h_{1}$ be a polynomial having a Siegel disk 
with center $z_{1}\in \CC .$ 
Let $h_{2}$ be a polynomial such that 
$z_{1}$ is a repelling fixed point of $h_{2}.$ 
Let $\G =\{ h_{1},h_{2}\} .$  
Let $f:\G \times \CCI \rightarrow \G \times \CCI $ be 
the skew product associated with the family $\G .$ 
Let $x =(h_{1},h_{1},h_{1},\ldots )\in \GN .$ 
Then, $(x,z_{1})\in \hat{J}^{x}(f)\setminus  J^{x}(f)$ and 
$z_{1}\in \hat{J}_{x}(f)\setminus J_{x}(f).$ 

 If $g$ is an open and surjective map (e.g. the shift map $\sigma :\G^{\NN }\rightarrow \G ^{\NN }$), 
 then $f^{-1}(\tilde{J}(f))=f(\tilde{J}(f))=\tilde{J}(f)$ (\cite[Lemma 2.4]{S1}). 
For more details, see \cite{S1,S4}.  
\item[(2)] 
Let $\Gamma $ be a compact subset of Rat and let $f:\Gamma ^{\NN }\times \CCI 
\rightarrow \Gamma ^{\NN }\times \CCI $ be the skew product associated with $\Gamma .$ 
Let $G$ be the rational semigroup generated by $\Gamma $ (thus $G=\{ g_{i_{1}}\circ \cdots \circ g_{i_{n}}\mid 
n\in \NN , \forall g_{i_{j}}\in \Gamma \} $). 
If $\sharp (J(G))\geq 3 $, then $\pi _{\CCI }(\tilde{J}(f))=J(G)$ (\cite[Lemma 3.5]{SdpbpIII}). 
From this result, we can apply the results of the dynamics of $f$ to the dynamics of $G.$  
\end{itemize}
\end{rem}
\begin{df}[\cite{SdpbpIII}]
Let $f:X\times \CCI \rightarrow X\times \CCI $ be a 
polynomial skew product over $g:X\rightarrow X.$ 
Then for each $x\in X$, we set
$K_{x}(f):=
\{ y\in \CCI  \mid 
\{ f_{x,n}(y)\} _{n\in \NN }
\mbox{ is bounded } $ in $\CC \}  $, 
 and 
 $A_{x}(f):=\{ y\in \CCI 
\mid f_{x,n}(y)\rightarrow \infty 
,\ n\rightarrow \infty \} .$
Moreover, we set 
$K^{x}(f):= \{ x\} \times K_{x}(f) \ (\subset 
X\times \CCI ) $ and 
$A^{x}(f):= \{ x\} \times A_{x}(f)\ (\subset X\times \CCI ).$ 
\end{df}

%\begin{rem}
%Let $f:\GN \times \CCI \rightarrow \GN \times \CCI 
%$ be the skew product associated with the family $\G $  
%of rational maps. Then 
%we have $F_{\g }(f)=
%\{ \g \} \times \{ y\in \CCI \mid 
%\{ \g_{n}\circ \cdots \circ \g_{1}\} _{n\in \NN }
%\mbox{ is normal in a neighborhood of } y \mbox{ in }\CCI \} $,\ 
%for each $\g =(\g _{1},\g _{2},\cdots )\in \GN .$
%\end{rem}
\begin{df}[\cite{SdpbpI}] 
Let $G$ 
be a 
polynomial semigroup generated by a 
%compact 
subset $\G $ of Poly$_{\deg \geq 2}.$ 
%Fix the 
%generator system $\{ h_{1},\ldots ,h_{m}\} .$ 
Suppose 
$G\in {\cal G}_{dis}.$ Then 
we set 
$ \G_{\min }:=\{ h\in \G \mid 
J(h)\subset J_{\min }\} ,$
where $J_{\min }$ denotes the 
unique minimal element in $({\cal J},\ \leq )$ 
in statement \ref{mainth2-3} of Theorem~\ref{mainth2}.  
Furthermore, if $\G _{\min }\neq \emptyset $, 
let $G_{\min ,\G }$ be the subsemigroup 
of $G$ that is generated by 
$\G _{\min }.$ 
\end{df}
\begin{rem}
\label{jminrem}
Let $G$ be a polynomial semigroup generated by a compact subset 
$\G $ of Poly$_{\deg \geq 2}.$ Suppose $G\in {\cal G}_{dis}.$ Then, 
 by Proposition~\ref{bminprop},  
%and Theorem~\ref{mainth0}, 
 we have $\G _{\min }\neq \emptyset $ and 
 $\G \setminus \G _{\min }\neq \emptyset .$ 
 Moreover, $\G _{\min }$ is a compact subset of $\G .$ For, 
 if $\{ h_{n}\} _{n\in \NN }\subset \G _{\min }$ and 
 $h_{n}\rightarrow h_{\infty }$ in $\G $, 
 then for each repelling periodic point $z_{0}\in 
 J(h_{\infty })$ of $h_{\infty }$, 
% the sets $\{ J(h_{n})\} _{n\in \NN }$ accumulates in $z_{0}$, which implies 
% that $h_{\infty }\in \G _{\min }.$  
we have that $d(z_{0}, J(h_{n}))\rightarrow 0$ as $n\rightarrow \infty $, 
which implies that $z_{0}\in J_{\min }$ and thus $h_{\infty }\in \G _{\min }.$ 
\end{rem}
\noindent {\bf Notation:} 
Let ${\cal F}:= \{ \varphi _{n}\} _{n\in \NN }$  
%sequence $\{ \varphi _{n}\} _{n\in \NN }$ 
be a sequence  
of meromorphic functions 
in a domain $V.$ We say that a meromorphic function 
$\psi $ is a limit function of ${\cal F}$ 
%$\{ \varphi _{n}\} _{n\in \NN }$ 
if there exists a strictly increasing sequence $\{ n_{j}\} _{j\in \NN }$ of 
positive integers 
%$\{ n_{j}\} _{j\in \NN }$ in $\NN $ 
%$\{ \varphi _{j}\} _{j\in \NN }$ in ${\cal F}$ 
such that 
$\varphi _{n_{j}}\rightarrow \psi $ locally uniformly on 
$V$, as $j\rightarrow \infty .$   
\begin{df}
Let $\G $ and $S$ be non-empty subsets of Poly$_{\deg \geq 2}$ 
with $S\subset \G .$ We set 
$$R(\G ,S):= 
\left\{ \g =(\g _{1},\g _{2}, \ldots )\in \GN \mid 
\sharp (\{ n\in \NN \mid \g _{n}\in S\} )=\infty \right\} .$$ 
\end{df}

\begin{df}
\label{d:hypskew}
Let $f:X\times \CCI 
\rightarrow X\times \CCI $ be a 
rational skew product over 
$g:X\rightarrow X.$ 
We set
%$$ C(f):= \bigcup _{x\in X}\{ v\in \pi ^{-1}\{ x\} 
%\mid v \mbox{ is a critical point of } f_{x}:
%\pi ^{-1}\{ x\} \rightarrow \pi ^{-1}\{ g(x)\} \} ,$$ 
%under the canonical identification of 
%$\pi ^{-1}\{ x\} \cong \pi ^{-1}\{ g(x)\} \cong \CCI .$ 
$$C(f):= \{ (x,y)\in X\times \CCI \mid y \mbox{ is a critical point of }
f_{x,1}\} .$$ 
Moreover, we set   
$P(f):=\overline{\cup _{n\in \NN  }f^{n}(C(f))}, $
%(\mbox{critical points of }f_{x} :\pi ^{-1}\{ x\} 
%\rightarrow \pi ^{-1}\{ g(x)\} ) } \ (\subset X\times \CCI )$$
%under the canonical identification 
%$\pi ^{-1}\{ x\} \cong \CCI ,\pi ^{-1}\{ g(x)\} \cong \CCI $, 
where the closure is taken in the product space $X\times \CCI .$ 
This $P(f)$ is called the {\bf fiber-postcritical set} of 
$f.$ 

 We say that $f$ is hyperbolic 
(along fibers) if 
$P(f)\subset F(f).$

\end{df}
We present a result which describes 
the details of the fiberwise dynamics along 
$\g $ in $R(\G ,\G \setminus \G _{\min }).$  
We recall that a Jordan curve $\xi $ in $\CCI $ is said to be 
a $K$-quasicircle, if $\xi $ is the image of $S^{1}(\subset \CC )$  
under a $K$-quasiconformal 
homeomorphism $\varphi :\CCI \rightarrow \CCI .$  
(For the definition of a quasicircle and a quasiconformal homeomorphism, see 
\cite{LV}.) 
\begin{thm}
\label{mainth3}
Let $G$ be a polynomial 
semigroup generated by a compact subset $\G $ of 
{\em Poly}$_{\deg \geq 2}.$  Suppose $G\in {\cal G}_{dis}.$ 
Let $f:\GN \times \CCI 
\rightarrow \GN \times \CCI $ be 
the  skew product associated with 
the family $\G $  
of polynomials. 
%Let 
%$S$ be a non-empty compact subset of $\G \setminus \G _{\min }.$ 
%Moreover, we set
%$$R_{S}:=\left\{ \g =(\g _{1},\g _{2},\cdots )\in \GN \mid 
%\sharp (\{ n\in \NN \mid \g _{n}\in S\} ) =\infty \right\} .$$   
Then, 
%under the above notation and 
%notation in Section 1,\ 
 all of the following statements 
\ref{mainth3-0},\ref{mainth3-1}, and \ref{mainth3-3} hold. 
\begin{enumerate}
\item \label{mainth3-0}
Let $\g \in R(\G ,\G \setminus \G _{\min }).$ 
Then, each limit function of 
$\{ f_{\g ,n}\} _{n\in \NN }$ in each connected component 
of $F_{\g }(f)$ is constant. 
\item \label{mainth3-1}
Let $S$ be a non-empty compact subset of $\G \setminus \G _{\min }.$ 
Then, for each $\g \in R(\G ,S)$, we have the following.
%\Sigma _{m}\setminus 
%\bigcup _{n\geq 0}\sigma ^{-n}B_{\min }^{\NN },\ $

 \begin{enumerate}
 
 \item \label{mainth3-1-1}
 There exists exactly one bounded component $U_{\g }$ 
 of $F_{\g }(f).$
 %$\pi ^{-1}\{ x\}  \setminus J_{x}(f)$
%  such that 
%  $\pi _{\CCI }(U_{\g })$ is a bounded component 
%  of $ \pi _{\CCI}(F_{\g }(f))\cap \CC .$ 
  Furthermore,\ 
 $\partial U_{\g }=\partial A_{\gamma }(f)=J_{\g }(f).$ 
% in $\pi ^{-1}\{ \g \} $ is equal to 
% $J_{\g }(f).$
 
 \item \label{mainth3-1-2}
% Each limit function of 
% $\{ f_{\g ,n} \} _{n}$ in 
% $U_{\g }$ is constant. Moreover, 
For each $y\in U_{\g }$,  there exists a number 
 $n\in \NN $ such that $f_{\g ,n}(y)
 \in $ {\em int}$(\hat{K}(G)).$
 
 \item \label{mainth3-1-3}
 $\hat{J}_{\g }(f)=J_{\g }(f).$ 
 Moreover, the map $\omega \mapsto J_{\omega }(f)$ defined on 
 $\GN $ 
 is continuous at $\g $, with respect to the Hausdorff metric  
 in the space of non-empty compact subsets of $\CCI .$ 
 \item \label{mainth3-1-4}
 The 2-dimensional Lebesgue measure of 
 $\hat{J}_{\g }(f)=J_{\g }(f)$
 is equal to zero. 
% under the natural identification 
% $\pi ^{-1}\{ \g \}  \cong \CCI .$  
 
 \end{enumerate}

\item 
\label{mainth3-3}
Let $S$ be a non-empty compact subset of 
$\G \setminus \G _{\min }.$ For each $p \in \NN ,\ $ 
%we set 
%$$W_{s}:=\{ x\in \Sigma _{m}\mid 
%\forall l\in \NN,\ 
%1\leq \exists j \leq s \mbox{ with }
%x_{l+j}\in \{ 1,\ldots ,m\} \setminus B_{\min }\} .$$ 
we denote by $W_{S,p}$ the set of 
elements $\g =(\g _{1},\g _{2},\ldots )\in \GN $ such that 
for each $l\in \NN $, at least one of 
$\g _{l+1},\ldots ,\g _{l+p}$ belongs to $S.$
Let $\overline{f}:=
f| _{W_{S,p}\times \CCI }:
W_{S,p}\times \CCI \rightarrow 
W_{S,p}\times \CCI .$ Then,  
$\overline{f} $ is a hyperbolic skew product 
over the shift map $\sigma :W_{S,p}\rightarrow W_{S,p}$,  
and there exists a constant 
$K_{S,p}\geq 1$ such that
for each $\g \in W_{S,p},\ $ 
$\hat{J}_{\g }(f)=J_{\g }(f)
=J_{\g }(\overline{f})$ is a $K_{S,p}$-quasicircle.  
\end{enumerate}
\end{thm}
\begin{df}
\label{d:hypsh}
Let $G$ be a rational semigroup. 
\begin{enumerate}
\item 
We say that 
$G$ is hyperbolic if $P(G)\subset F(G).$ 
\item We say that $G$ is semi-hyperbolic if 
there exists a number $\delta >0$ and a 
number $N\in \NN $ such that 
for each $y\in J(G)$ and each $g\in G$, 
we have $\deg (g:V\rightarrow B(y,\delta ))\leq N$ for 
each connected component $V$ of $g^{-1}(B(y,\delta ))$, 
where $B(y,\delta )$ denotes the ball of radius $\delta $ 
with center $y$ with respect to the spherical distance, 
and $\deg (g:\cdot \rightarrow \cdot )$ denotes the 
degree of finite branched covering. 
(For background on semi-hyperbolicity, see \cite{S1} and \cite{S4}.) 
\end{enumerate}
\end{df}

\begin{thm}
\label{mainth3-2}
%If $H_{\min }$ is semi-hyperbolic, then 
%$G$ is also. 
%If $G$ is semi-hyperbolic, then 
%for each $\g \in R_{S}$, 
%\Sigma _{m}\setminus 
%\bigcup _{n\geq 0}\sigma ^{-n}B_{\min }^{\NN },\ $
% we have that 
% $\hat{J}_{\g }(f)=J_{\g }(f)$
%  is a Jordan curve. 
%Here,\ a rational semigroup $H$ is said to be 
%semi-hyperbolic if for each $z\in J(H)$ there exists a 
%neighborhood $U$ of $z$ in $\CCI $ and a 
%number $N\in \NN $ such that 
%for each $g\in G,\ $ $\deg (g:V\rightarrow U)\leq N$ 
%for each connected component $V$ of $g^{-1}(U).$ 
%(For background of semi-hyperbolicity, see \cite{S1} and \cite{S4}.) 
Let $G$ be a polynomial semigroup generated by a compact 
subset $\G $ of {\em Poly}$_{\deg \geq 2}.$ 
Let $f:\GNCR $ be the skew product associated with 
the family $\G .$ Suppose that $G\in {\cal G}_{dis}$ and that 
$G$ is semi-hyperbolic. Let $\g \in R(\G , \G \setminus \G _{\min })$ be 
any element. Then, $\hat{J}_{\g }(f)=J_{\g }(f)$ and 
$J_{\g }(f)$ is a Jordan curve. Moreover, 
for each point $y_{0}\in $ {\em int}$(K_{\g }(f))$, there exists an $n\in \NN $ 
such that $f_{\g ,n}(y_{0})\in $ {\em int}$(\hat{K}(G)).$

\end{thm}
We next present a result which states that there exist    
 families of uncountably many mutually disjoint quasicircles 
with uniform distortion, densely inside the Julia set of 
a semigroup in ${\cal G}_{dis}.$  
\begin{thm}
\label{cantorqc}
{\bf (Existence of a Cantor family of quasicircles.)} 
Let $G\in {\cal G}_{dis}$ (possibly generated by a non-compact 
family) 
and let $V$ be an open subset of $\CCI $ 
with $V\cap J(G)\neq \emptyset. $  
Then, there exist elements $g_{1}$ and 
 $g_{2}$ in $G$ such that all of the following hold.
 \begin{enumerate}
 \item 
 \label{cantorqc1}
 $H=\langle g_{1},g_{2}\rangle $ satisfies 
 that $J(H)\subset J(G).$ 
 \item 
 \label{cantorqc2}
 There exists a non-empty open set $U$ in $\CCI $ such that    
 $g_{1}^{-1}(\overline{U})\cup g_{2}^{-1}(\overline{U}) 
 \subset U$, and such that $g_{1}^{-1}(\overline{U})
 \cap g_{2}^{-1}(\overline{U})=\emptyset .$  
 
 \item 
 \label{cantorqc3}
 $H=\langle g_{1},g_{2}\rangle $ is a 
 hyperbolic polynomial semigroup.
% (i.e. 
% $ P(H)\subset F(H)$).
 
 \item  
 \label{cantorqc4}
 Let 
 $f:\GN \times \CCI \rightarrow 
 \GN \times \CCI $ be the skew product  
 associated with the family  
 $\G= \{ g_{1},g_{2}\} $ of polynomials. Then, 
 we have the following. 
 \begin{enumerate}
 \item \label{cantorqc4a}
 $J(H)=\bigcup _{\g \in \GN }
 J_{\g }(f)$ (disjoint union). 
 Each $J_{\g }(f)$ is connected and 
 $(\{ J_{\g }(f)\} _{\g \in \GN},\leq )$ is totally ordered. 
 \item \label{cantorqc4b}
 For each connected component $J$ of $J(H)$, there exists an 
 element $\g \in \GN $ such that   
 $J=J_{\g }(f).$  
  
 \item \label{cantorqc4c}
 There exists a constant $K\geq 1$
 independent of $J$ 
  such that each connected component $J$ of $J(H)$ 
  is a $K$-quasicircle. 
  \item 
 \label{cantorqc5}
 The map $\g \mapsto J_{\g }(f)$, defined for all 
 $\g \in \GN $, is continuous with respect to 
 the Hausdorff metric in the space of non-empty compact subsets of 
 $\CCI $, and 
 injective.   
 \item 
 \label{cantorqc6}
For each element $\g \in \GN ,$  
 $J_{\g }(f)\cap V\cap J(G)\neq \emptyset .$ 
\item 
\label{cantorqc7}
Let $\omega \in \GN $ be an element such that 
$\sharp (\{ j\in \NN \mid \omega _{j}=g_{1}\} )=\infty $ 
and such that $\sharp (\{ j\in \NN \mid \omega _{j}=g_{2}\} )=\infty .$ 
Then, 
 $J_{\omega }(f)$ does not meet the boundary of any connected component of 
 $F(G).$  
 \end{enumerate} 
\end{enumerate}
\end{thm}
\begin{rem}
\label{r:cantorf}
This ``Cantor family of quasicircles'' in the research of rational semigroups was introduced by the author of this paper. 
By using this idea, in \cite{SS} (which was written after this paper), 
it is shown that for a polynomial semigroup $G\in {\cal G}_{dis}$ which is generated by 
a (possibly non-compact) family of Poly$_{\deg \geq 2}$, if $A$ and $B$ are two different 
doubly connected components of $F(G)$, then there exists a Cantor family ${\cal C}$ of quasicircles in $J(G)$ 
such that each element of ${\cal C}$ separates $\overline{A}$ and $\overline{B}$.      
In Theorem~\ref{cantorqc} of this paper, 
we show that there exist Cantor families of quasicircles {\em densely} inside the Julia set of a semigroup 
$G\in {\cal G}_{dis}$, which is of independent value. 
\end{rem}
\subsection{Fiberwise Julia sets that are Jordan curves but not quasicircles}
\label{fjjq}
We present a result on a sufficient condition for a fiberwise Julia set 
$J_{x}(f)$ to satisfy that $J_{x}(f)$ is  a Jordan curve but not a quasicircle, 
the unbounded component of $\CCI \setminus J_{x}(f)$ is  a John domain, and 
the bounded component of $\CC \setminus J_{x}(f)$ is not a John domain. 
Note that we have many examples of this phenomenon (see Proposition~\ref{Constprop},Remark~\ref{r:manyex},Example~\ref{jbnqexfirst}), 
and note also that 
this phenomenon cannot hold in the usual iteration dynamics of a single polynomial map 
$g$ with $\deg (g)\geq 2$ (see Remark~\ref{r:jbnq}).   
The proofs are given in Section~\ref{Proofs of fjjq}.
\begin{df}
Let $V$ be a subdomain of $\CCI $ such that 
$\partial V\subset \CC .$  
We say that $V$ is a John domain if there exists a 
constant $c>0$ and a point $z_{0}\in V$ ($z_{0}=\infty $ when 
$\infty \in V$) satisfying the following:  
for all $z_{1}\in V$ there exists an arc $\xi \subset V$ connecting 
$z_{1}$ to $z_{0}$ such that 
for any $z\in \xi $, we have 
$\min \{ |z-a|\mid a\in \partial V\} \geq c|z-z_{1}|.$
\end{df}
\begin{rem}
Let $V$ be a simply connected domain in $\CCI $ such that 
$\partial V\subset \CC .$ 
It is well-known that  
if $V$ is a John domain, then 
$\partial V$ is locally connected (\cite[page 26]{NV}). 
Moreover, a Jordan curve $\xi \subset \CC $ is a 
quasicircle if and only if both components of $\CCI \setminus \xi $ are 
John domains (\cite[Theorem 9.3]{NV}).  
\end{rem}
\begin{thm}
\label{mainthjbnq}
Let $G$ be a polynomial 
semigroup generated by a compact subset $\G $ of 
{\em Poly}$_{\deg \geq 2}.$  Suppose that $G\in {\cal G}_{dis}.$ 
Let $f:\GN \times \CCI 
\rightarrow \GN \times \CCI $ be 
the  skew product associated with 
the family $\G $  
of polynomials. 
%Let 
%$S$ be a non-empty compact subset of $\G \setminus \G _{\min }.$ 
%Moreover, we set 
%$$R_{S}:=\left\{ \g =(\g _{1},\g _{2},\ldots )\in \GN \mid 
%\sharp (\{ n\in \NN \mid \g _{n}\in S\} )=\infty \right\} .$$   
Let $m\in \NN $ and suppose that there exists an element 
$(h_{1},h_{2},\ldots ,h_{m})\in \G ^{m}$  
such that $J(h_{m}\circ \cdots \circ h_{1})$ is not a quasicircle.
Let $\alpha =(\alpha _{1},\alpha _{2},\ldots )\in \GN $ be the  
element such that for each $k,l\in \NN \cup \{ 0\} $ with $1\leq l\leq m$, 
$\alpha _{km+l}=h_{l}.$ 
%Moreover, let $\g \in R_{S} $ be an element such that there exists a 
%sequence $(n_{k}) $ of integers satisfying that  
%$\sigma ^{n_{k}}(\g )\rightarrow \alpha $ as $k\rightarrow \infty .$ 
%Then, under the notation of Theorem~\ref{mainth3}, 
%we have the following. 
%\begin{enumerate}
%\item $J_{\g }(f)$ is a Jordan curve, but not a quasicircle.
%\item The unbounded component 
%$A_{\g }(f)$ of $F_{\g }(f)$ is a John domain, but the 
%unique bounded component $U_{\g }$ of $F_{\g }(f)$ 
%is not a John domain.
%, under the natural identification $\pi ^{-1}\{ \g \} 
%\cong \CCI .$ 
%\end{enumerate}
Then, the following statements \ref{mainthjbnq1} and 
\ref{mainthjbnq2} hold.
\begin{enumerate}
\item \label{mainthjbnq1}
Suppose that $G$ is hyperbolic. 
Let $\g \in R(\G ,\G \setminus \G _{\min })$ be an element 
such that there exists a sequence $\{ n_{k}\} _{k\in \NN }$ 
of positive integers satisfying that 
$\sigma ^{n_{k}}(\g )\rightarrow \alpha $ as $k\rightarrow \infty .$ 
Then, $J_{\g }(f)$ is a Jordan curve but not a quasicircle. 
Moreover, the unbounded component $A_{\g }(f)$ of 
$F_{\g }(f)$ is a John domain, but the unique bounded component 
$U_{\g }$ of $F_{\g }(f)$ is not a John domain.
\item \label{mainthjbnq2}
Suppose that $G$ is semi-hyperbolic. 
Let $\rho _{0}\in 
\G \setminus \G _{\min }$ be any element and 
let $\beta := (\rho _{0},\alpha _{1},\alpha _{2},\ldots )\in \GN .$ 
Let $\g \in R(\G ,\G \setminus \G _{\min })$ be an element 
such that there exists a sequence $\{ n_{k}\} _{k\in \NN }$ 
of positive integers satisfying that 
$\sigma ^{n_{k}}(\g )\rightarrow \beta $ as $k\rightarrow \infty .$ 
Then, $J_{\g }(f)$ is a Jordan curve but not a quasicircle. 
Moreover, the unbounded component $A_{\g }(f)$ of 
$F_{\g }(f)$ is a John domain, but the unique bounded component 
$U_{\g }$ of $F_{\g }(f)$ is not a John domain.

\end{enumerate}
\end{thm} 
We now classify hyperbolic two-generator polynomial semigroups in ${\cal G}_{dis}.$ 
Moreover, we completely classify the fiberwise Julia sets $J_{\g }(f)$ in terms of the information on $\g .$ 
\begin{thm}
\label{t:tghyp}
Let $\G =\{ h_{1},h_{2}\} \subset \mbox{{\em Poly}}_{\deg \geq 2}.$ Let $G=\langle h_{1},h_{2}\rangle .$ 
Suppose $G\in {\cal G}_{dis}$ and that $G$ is hyperbolic. 
Let $f:\GN \times \CCI \rightarrow \GN \times \CCI $ be the skew product associated with $\G .$ 
Then, 
for each connected component $J$ of $J(G)$, there exists a unique $\g \in \GN $ such that 
$J=J_{\g }(f).$ Moreover, exactly one of the following statements 1, 2 holds.
\begin{enumerate}
\item There exists a constant $K\geq 1$ such that for each $\g \in \GN $, $J_{\g }(f)$ is a $K$-quasicircle. 
\item There exists a unique $j\in \{ 1,2\} $ such that $J(h_{j})$ is not a Jordan curve. In this case, 
for each $\g =(\g _{1},\g _{2},\ldots )\in \GN $, exactly one of the following statements {\em (a),(b), (c)} holds. 
\begin{itemize}
\item[{\em (a)}] There exists a $p\in \NN $ such that for each 
$l\in \NN $, at least one of $\g _{l+1},\ldots ,\g _{l+p}$ is not equal to $h_{j}.$ 
Moreover, $J_{\g }(f)$ is a quasicircle. 
\item[{\em (b)}] $\sharp \{ n\in \NN \mid \g _{n}\neq h_{j}\} =\infty $ and 
there exists a strictly increasing sequence $\{ n_{k}\} _{k\in \NN }$ in $\NN $ such that 
$\sigma ^{n_{k}}(\g )\rightarrow (h_{j},h_{j},h_{j},\ldots )$ as $k\rightarrow \infty .$ 
Moreover, $J_{\g }(f)$ is a Jordan curve but not a quasicircle, the unbounded component 
$A_{\g }(f)$ of $\CCI \setminus J_{\g }(f)$ is a John domain, and the bounded component of 
$\CC \setminus J_{\g }(f)$ is not a John domain. 
\item[{\em (c)}] There exists an $l\in \NN $ such that 
$\sigma ^{l}(\g )=(h_{j},h_{j},h_{j},\ldots )$. 
Moreover, $J_{\g }(f)$ is not a Jordan curve.  
\end{itemize}

\end{enumerate}

\end{thm}
\subsection{Random dynamics of polynomials}
\label{random}
In this section, we present some results on 
the random dynamics of polynomials. 
 The proofs are given in Section~\ref{Proofs of random}.
 
 Let $\tau $ be a Borel probability measure on Poly$_{\deg \geq 2}.$  
%of which support in Poly$_{\deg \geq 2} $ is compact. 
We consider the i.i.d. random dynamics on $\CCI $ such that,  
at every step, we choose a polynomial map 
$h:\CCI \rightarrow \CCI $ according to the distribution $\tau .$ 
(Hence, this defines a kind of Markov process on $\CCI $ such that,  
at every step, the transition probability $p(x,A)$ 
from a point $x\in \CCI $ to a Borel subset $A$ of $\CCI $ is equal to $\tau (\{ h\in \mbox{Poly}_{\deg \geq 2}\mid 
h(x)\in A\} )$.) \\ 
{\bf Notation:} 
For a Borel probability measure $\tau $ on  Poly$_{\deg \geq 2}$, 
we denote by $\G_{\tau }$ the topological support of $\tau $ on 
Poly$_{\deg \geq 2}.$ (Hence, $\G_{\tau }$ is a closed set in 
Poly$_{\deg \geq 2}.$) 
%We set $\G _{\tau }:= $ supp $\tau .$  
Moreover, we denote by $\tilde{\tau } $ the infinite product measure $\otimes _{j=1}^{\infty } \tau .$  
This is a Borel probability measure on 
$\G_{\tau }^{\NN }.$ 
% Furthermore, we denote by $G_{\tau }$ the polynomial semigroup 
%generated by $\G _{\tau }.$  
\begin{df}
Let $X$ be a complete metric space. 
A subset $A$ of $X$ is said to be residual if 
$X\setminus A$ is a countable union of nowhere dense subsets of 
$X.$ Note that by Baire Category Theorem, a residual set $A$ is dense in $X.$ 
\end{df}
\begin{cor}\label{rancor1}
(Corollary of Theorem~\ref{mainth3}-\ref{mainth3-1})
%Let $\tau $ be a Borel probability measure 
Let $\G $ be a non-empty compact subset of 
{\em Poly}$_{\deg \geq 2}.$ 
Let 
$f:\G ^{\NN }\times \CCI \rightarrow 
\G ^{\NN }\times \CCI $ be the skew product 
associated with the family $\G $ of polynomials.
Let $G$ be the polynomial semigroup generated by 
$\G .$ 
Suppose $G\in {\cal G}_{dis}.$ 
Then, there exists a residual subset ${\cal U}$ of $\G ^{\NN }$  
such that for each Borel probability measure 
$\tau $ on {\em Poly}$_{\deg \geq 2}$ with $\G _{\tau }=\G $, 
we have  
 $\tilde{\tau }({\cal U})=1$, and such that 
each $\g \in {\cal U}$ satisfies   
all of the following.
\begin{enumerate}
 \item \label{rancor1-1}
 There exists exactly one bounded component $U_{\g }$ 
 of $F_{\g }(f).$
 %$\pi ^{-1}\{ x\}  \setminus J_{x}(f)$
%  such that 
%  $\pi _{\CCI }(U_{\g })$ is a bounded component 
%  of $ \pi _{\CCI}(F_{\g }(f))\cap \CC .$ 
  Furthermore,
%\ the boundary 
 $\partial U_{\g }=\partial A_{\gamma }(f)=J_{\g }(f).$ 
% in $\pi ^{-1}\{ \g \} $ is equal to 
% $J_{\g }(f).$
 
 \item \label{rancor1-2}
 Each limit function of 
 $\{ f_{\g ,n} \} _{n}$ in 
 $U_{\g }$ is constant. Moreover, 
 for each $y\in U_{\g }$,  there exists a number 
 $n\in \NN $ such that $f_{\g ,n}(y)
 \in $ {\em int}$(\hat{K}(G)).$
 
 \item \label{rancor1-3}
 We have $\hat{J}_{\g }(f)=J_{\g }(f).$ Moreover, 
 the map $\omega \mapsto J_{\omega }(f)$ defined on 
 $\GN $ is continuous at 
 $\g $, with respect to the Hausdorff metric in the space of 
 non-empty compact subsets of $\CCI .$ 
 \item \label{rancor1-4}
The 2-dimensional Lebesgue measure of 
 $\hat{J}_{\g }(f)=J_{\g }(f)$
 is equal to zero.
% , under the natural identification 
% $\pi ^{-1}\{ \g \}  \cong \CCI .$  

\end{enumerate}

\end{cor}
\begin{cor}
\label{c:rancor2sh}
%\begin{cor}\label{rancor1}
(Corollary of Theorems~\ref{mainth3-2}, \ref{mainthjbnq})
%Let $\tau $ be a Borel probability measure 
Let $\G $ be a non-empty compact subset of 
{\em Poly}$_{\deg \geq 2}.$ 
Let 
$f:\G ^{\NN }\times \CCI \rightarrow 
\G ^{\NN }\times \CCI $ be the skew product 
associated with the family $\G $ of polynomials.
Let $G$ be the polynomial semigroup generated by 
$\G .$ 
Suppose $G\in {\cal G}_{dis}$ and $G$ is semi-hyperbolic. 
Then, we have both of the following.
\begin{enumerate}
\item \label{c:rancor2sh1}
There exists a residual subset ${\cal U}$ of $\G ^{\NN }$  
such that, for each Borel probability measure 
$\tau $ on {\em Poly}$_{\deg \geq 2}$ with $\G _{\tau }=\G $, 
we have  
 $\tilde{\tau }({\cal U})=1$, and such that,  
for each $\g \in {\cal U}$ and for each point $y_{0}\in \mbox{{\em int}}(K_{\g }(f))$,  
$J_{\g }$ is a Jordan curve and there exist an $n\in \NN $ with $f_{\g, n}(y_{0})\in \mbox{{\em int}}(\hat{K}(G)).$  
\item \label{c:rancor2sh2}
Suppose further that there exists an element $h\in G$ such that 
$J(h)$ is not a quasicircle. 
Then, there exists a residual subset ${\cal V}$ of $\G ^{\NN }$  
such that, for each Borel probability measure 
$\tau $ on {\em Poly}$_{\deg \geq 2}$ with $\G _{\tau }=\G $, 
we have  
 $\tilde{\tau }({\cal V})=1$, and such that,  
for each $\g \in {\cal V}$,     
$J_{\g }$ is a Jordan curve but not a quasicircle, the unbounded component of 
$\CCI \setminus J_{\g }$ is a John domain and the bounded component of 
$\CC \setminus J_{\g }$ is not a John domain. 
\end{enumerate} 
\end{cor}
\begin{rem}
\label{r:jbnq}
Let $h\in $ Poly$_{\deg \geq 2}$.  
Suppose that $J(h)$ is a Jordan curve but not a quasicircle. 
Then, it is easy to see that there exists a parabolic fixed point 
of $h$ in $\CC $ and the bounded connected component $U$ of $F(h)$ is the 
immediate parabolic basin. (In fact, 
we have $h^{-1}(U)=U=h(U)$ and $U$ is the immediate basin of 
either attracting or parabolic fixed point of $h.$ 
If $U$ is the immediate basin of an attracting fixed point of $h$, 
then by using quasiconformal surgery (e.g. \cite[Theorem 4.1]{SdpbpIII}), 
we obtain that $J(h)$ is a quasicircle. However, this is a contradiction.)  
Hence, $\langle h\rangle $ is not semi-hyperbolic (see \cite{CJY}). 
Moreover, by \cite{CJY}, 
$F_{\infty }(h)$ is not a John domain. 

% The argument above means that 
%the following situation: \\ 
%`` for almost every sequence $\g $, 
%the Julia set $J_{\g }(f) $ is a Jordan curve but not a quasicircle, 
%and the basin of infinity $A_{\g }(f)$ is a John domain'' \\ 
%(cf. statement \ref{mainthran1-2} in Theorem~\ref{mainthran1} and 
%statement \ref{mainthran2-2} in Theorem~\ref{mainthran2} )  is a 
%special phenomenon in the random dynamics of polynomials, which 
%does not occur in the usual dynamics of polynomials. 
% the situations \ref{mainthran1-2} 
% in Theorem~\ref{mainthran1} and  \ref{mainthran1-2} 
% in Theorem~\ref{mainthran2} are the special ones in  
% random dynamics of polynomials. 
 Thus what we see in Theorem~\ref{mainthjbnq} and statement \ref{c:rancor2sh2} of Corollary~\ref{c:rancor2sh}, 
 as illustrated in Example ~\ref{jbnqexfirst}, is a new and unexpected  
 phenomenon which can hold in the {\em random} dynamics 
 of a family of polynomials, but cannot hold in the usual 
 iteration dynamics of a single polynomial. 
 Namely, it can hold that for almost every $\g \in \GN $, 
$J_{\g }(f)$ is a Jordan curve and fails to be a quasicircle 
 while the basin of infinity $A_{\g }(f)$ is  a John domain. 
Whereas, if $J(h)$, for some polynomial $h$, is a Jordan curve which 
fails to be a quasicircle, then the basin of infinity $F_{\infty }(h)$ 
is necessarily {\bf not} a John domain.   

 Pilgrim and Tan Lei (\cite{PT}) showed that there exists a  
hyperbolic rational map $h$ with {\em disconnected} Julia set such that  
``almost every'' connected component of $J(h)$ 
is a Jordan curve but not a quasicircle.  
\end{rem}
\subsection{Examples}
\label{Const}
We give some  
%of construction of examples of 
 examples of semigroups $G$ in ${\cal G}_{dis}.$ The following proposition was proved in \cite{SdpbpI}. 
%elements in ${\cal G}_{dis}.$ 
\begin{prop}[\cite{SdpbpI}] 
\label{Constprop}
Let $G$ be a 
polynomial semigroup generated by 
a compact subset $\G $ of {\em Poly}$_{\deg \geq 2}.$ 
Suppose that $G\in {\cal G}$ and  
{\em int}$(\hat{K}(G))\neq \emptyset .$ 
Let $b\in $ {\em int}$(\hat{K}(G)).$ 
%be a point. 
Moreover, let $d\in \NN $ be any positive integer such that 
$d\geq 2$, and such that 
$(d, \deg (h))\neq (2,2)$ for each $h\in \G .$ 
Then, there exists a number $c>0$ such that,  
for each $a\in \CC $ with $0<|a|<c$, 
there exists a compact neighborhood $V$ of 
$g_{a}(z)=a(z-b)^{d}+b$ in {\em Poly}$_{\deg \geq 2}$ satisfying 
that, for any non-empty subset $V'$ of $V$,  
the polynomial semigroup 
$H_{\G, V'} $ generated by the family $\G \cup V'$ 
belongs to ${\cal G}_{dis}$, $\hat{K}(H_{\G, V'})=\hat{K}(G)$ 
 and $(\G \cup V')_{\min }\subset \G .$ 
Moreover, in addition to the assumption above, 
if $G$ is semi-hyperbolic (resp. hyperbolic), 
then the above $H_{\G, V'}$ is semi-hyperbolic (resp. hyperbolic).   
\end{prop}
%For the proof of this proposition, see \cite{SdpbpI}. 
\begin{rem}
\label{r:manyex}
By Proposition~\ref{Constprop}, 
there exists a $2$-generator polynomial semigroup 
$G=\langle h_{1},h_{2}\rangle $ in ${\cal G}_{dis}$ such that 
$h_{1}$ has a Siegel disk. Moreover, by Proposition~\ref{Constprop}, 
we can easily construct many examples of $G$ that satisfies statements~\ref{mainthjbnq1}, \ref{mainthjbnq2} 
of Theorem~\ref{mainthjbnq} and 
statement~\ref{c:rancor2sh2} of Corollary~\ref{c:rancor2sh}. 
\end{rem}
\begin{rem}
\label{r:gdisstar}
There are many ways to construct (semi-)hyperbolic semigroups $G\in {\cal G}_{dis}.$ 
For a $G\in {\cal G}_{dis}$ generated by a compact subset $\Gamma $ of Poly$_{\deg \geq 2}$, 
let $G_{\min ,\Gamma }$ be the polynomial semigroup generated by 
$\{ h\in \Gamma \mid J(h)\subset J_{\min }(G)\} .$ Then we have the following.
(1)(\cite[Theorem 2.36]{SdpbpI}) If $G_{\min ,\Gamma }$ is semi-hyperbolic, then $G$ is semi-hyperbolic. 
(2)(\cite[Theorem 2.37]{SdpbpI}) If $G_{\min ,\Gamma }$ is hyperbolic and 
$J_{\min }(G)\cap \bigcup _{h\in \Gamma \setminus \Gamma _{\min }}\mbox{CV}^{\ast }(h)=\emptyset $, 
then $G$ is hyperbolic. 

\end{rem}
\begin{ex}
\label{jbnqexfirst}
Let $g_{1}(z):=z^{2}-1$ and $g_{2}(z):= \frac{z^{2}}{4}.$ 
Let $\G :=\{ g_{1}^{2}, g_{2}^{2}\} .$ 
Moreover, let $G$ be the polynomial semigroup generated by 
$\G .$ Let $D:= \{ z\in \CC \mid |z|<0.4\} .$ Then, 
it is easy to see $g_{1}^{2}(D)\cup g_{2}^{2}(D)\subset D.$ Hence, 
$D\subset F(G).$ Moreover, by Remark~\ref{pcbrem}, we have that 
$P^{\ast }(G)=
\overline{\cup _{g\in G\cup \{ Id\} }g(\{ 0,-1\} )} 
\subset D\subset F(G).$ Hence, $G\in {\cal G}$ and $G$ is hyperbolic.  
Furthermore, let $K:=\{ z\in \CC \mid 0.4\leq |z|\leq 4\} .$ Then, 
it is easy to see that $(g_{1}^{2})^{-1}(K)\cup (g_{2}^{2})^{-1}(K)\subset K$ and 
$(g_{1}^{2})^{-1}(K)\cap (g_{2}^{2})^{-1}(K)=\emptyset .$ 
Combining these facts with \cite[Corollary 3.2]{HM1} and \cite[Lemma 2.4]{S3}, 
%Lemma~\ref{hmslem}-\ref{backmin} and Lemma~\ref{hmslem}-\ref{bss}, 
we obtain that $J(G)$ is disconnected. Therefore, 
$G\in {\cal G}_{dis}.$ 
Moreover, it is easy to see that 
$\G _{\min }=\{ g_{1}^{2}\} .$ 
Since $J(g_{1}^{2})$ is not a Jordan curve, 
we can apply Theorem~\ref{mainthjbnq}. Setting 
%$S:= \{ (g_{2}^{2},g_{2}^{2},g_{2}^{2},\ldots )\} $ and 
$\alpha := (g_{1}^{2},g_{1}^{2},g_{1}^{2},\ldots )\in \GN $, 
it follows that 
for any $$\g \in {\cal I}:= \left\{ \omega \in R(\G ,\G \setminus \G _{\min }) 
\mid 
\exists (n_{k}) \mbox { with } \sigma ^{n_{k}}(\omega )\rightarrow 
\alpha \right\} ,$$ 
$J_{\g }(f)$ is a Jordan curve but not a quasicircle, and  
$A_{\g }(f)$ is a John domain but the bounded component 
of $F_{\g }(f)$ is not a John domain. 
(See Figure~\ref{fig:dcgraph}: the Julia set of $G$. In this example, 
we have $(g_{1}^{2})^{-1}(J(G))\cap (g_{2}^{2})^{-1}(J(G))=\emptyset $, and so   
$\hat{{\cal J}}_{G}=\{ J_{\g }(f)\mid \g \in \GN \} $,  
and if $\g \neq \omega , J_{\g }(f)\cap J_{\omega }(f)=\emptyset .$) 
Note that by Theorem~\ref{t:tghyp}, 
%~\ref{mainth3}-\ref{mainth3-3}, 
if $\gamma \not\in {\cal I}$, 
then either $J_{\gamma }(f)$ is not a Jordan curve or $J_{\gamma }(f)$ is a quasicircle.
%\end{ex}
%\newpage 
\begin{figure}[htbp]
%\vspace*{5cm}
\caption{The Julia set of $G=\langle g_{1}^{2},g_{2}^{2}\rangle .$}
\ \ \ \ \ \ \ \ \ \ \ \ \ \ \ \ \ \ \ \ \ \ \ \ \ \ \ \ \ \ \ \ \ \ \ \ \ \ 
\ \ \ \ \ \ \ \ \ \ \ \ \ 
\includegraphics[width=3cm,width=3cm]{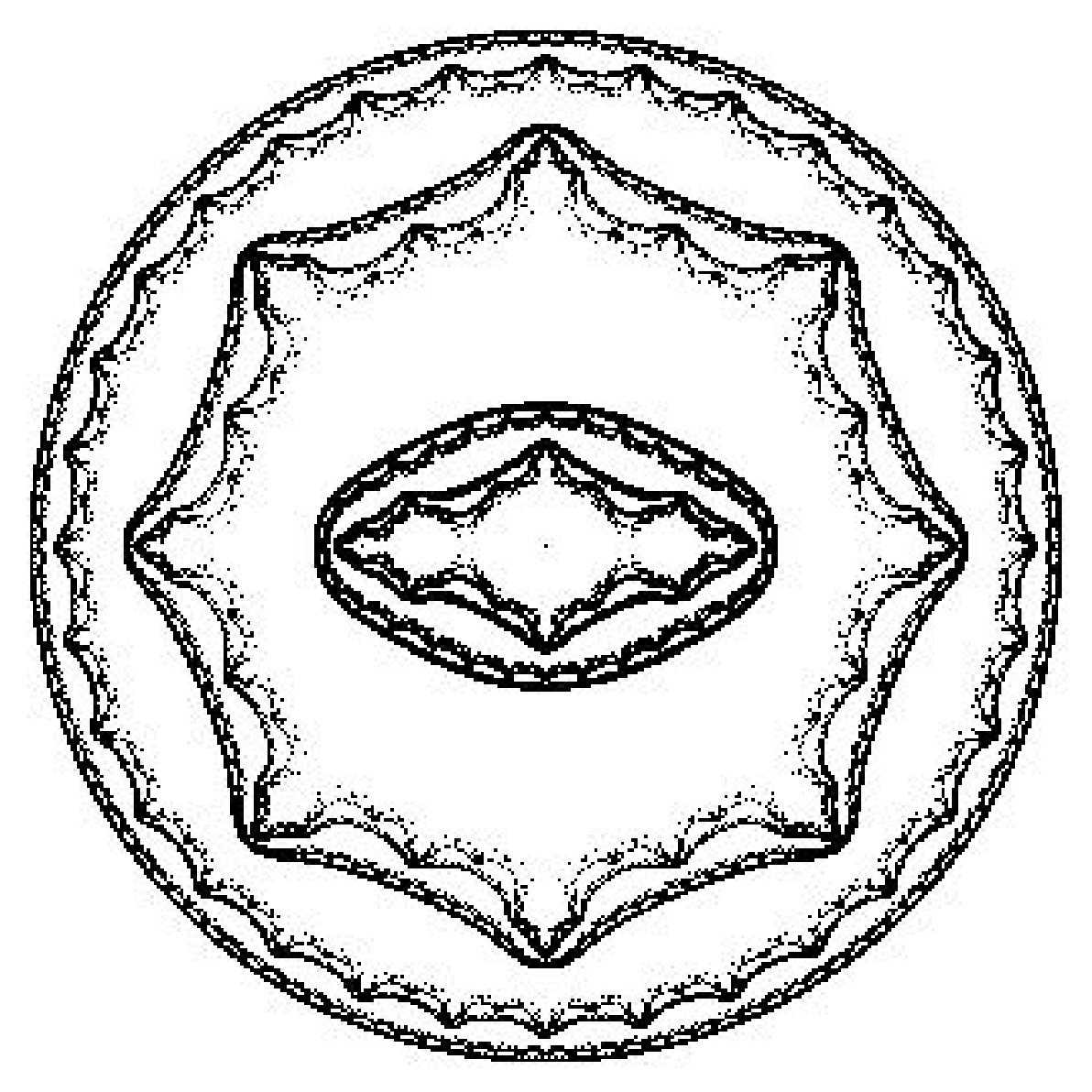}
\label{fig:dcgraph}
\end{figure}
\end{ex}
In \cite{SdpbpI, SdpbpIII, S7, SS, SU1, Splms}, we obtain many examples of postcritically bounded polynomial semigroups with 
many additional properties.  
In fact, several systematic ways to give such examples are found in those papers. 
\section{Tools}
\label{Tools}
In this section, we recall some fundamental tools to prove the main results. 

Let $G$ be a rational semigroup. Then, for each $g\in G$, $g(F(G))\subset F(G), g^{-1}(J(G))\subset J(G).$ 
If $G$ is generated by a compact family $\Lambda $ of Rat, then 
$J(G)=\bigcup _{h\in \Lambda }h^{-1}(J(G))$ (this is called backward self-similarity). 
If $\sharp J(G)\geq 3$, then $J(G)$ is a perfect set and 
$J(G)$ is equal to the closure of the set of repelling cycles of elements of $G$. 
We set $E(G):=\{ z\in \CCI \mid \sharp \bigcup _{g\in G}g^{-1}(\{ z\} )<\infty \} .$ 
If $\sharp J(G)\geq 3$, then $\sharp E(G)\leq 2$ and 
for each $z\in J(G)\setminus E(G)$, $J(G)=\overline{\bigcup _{g\in G}g^{-1}(\{ z\} ) }.$ 
If $\sharp J(G)\geq 3$, then $J(G)$ is the smallest set in 
$\{ \emptyset \neq K\subset \CCI \mid K \mbox{ is compact}, \forall g\in G, g(K)\subset K\} .$ For more details on these  
properties of rational semigroups, 
see \cite{HM1,St, GR, S3}.  For the dynamics of postcritically bounded polynomial semigroups, 
see \cite{SdpbpI, SdpbpIII, SS}. 
If $f:X\times \CCI \rightarrow X\times \CCI $ is a polynomial skew product such that 
$\deg (f_{x})\geq 2$ for each $x\in X$ 
and such that 
$\pi _{\CCI }(P(f))\setminus \{ \infty \} $ is bounded in $\CC $, 
then for each $x\in X$, 
$J_{x}(f)$ is connected (\cite[Lemma 3.6]{SdpbpIII}). 
For some fundamental properties of skew products, see 
\cite{S1,S4,SdpbpIII}. 

\section{Proofs}
\label{Proofs}
In this section, we give the proofs of the main results. 

\subsection{Proofs of the results in \ref{fibsec}}
\label{pffibsec}
In this section, we prove results in section \ref{fibsec}.

To prove results in \ref{fibsec}, we need the following notations and 
lemmas.
\begin{df}[\cite{S1}]
Let $f:X\times \CCI \rightarrow X\times \CCI $ be a rational skew 
product over $g:X\rightarrow X.$ Let $N\in \NN .$ 
We say that a point $(x_{0},y_{0})\in X\times \CCI $ belongs to 
$SH_{N}(f)$ if there exists a neighborhood $U$ of 
$x_{0}$ in $X$ and a positive number $\delta $ such that,  
for any $x\in U$, any $n\in \NN $, any $x_{n}\in g^{-n}(x)$, 
and any connected component 
$V$ of $(f_{x_{n},n})^{-1}(B(y_{0},\delta ))$, we have 
$\deg (f_{x_{n},n}:V\rightarrow B(y_{0},\delta ))\leq N.$ 
Moreover, we set 
$UH(f):= (X\times \CCI )\setminus \cup _{N\in \NN }SH_{N}(f).$ 
We say that $f$ is semi-hyperbolic (along fibers) if 
$UH(f)\subset \tilde{F}(f).$ 
\end{df}
\begin{rem}
Under the above notation, we have $UH(f)\subset P(f).$ 
\end{rem}
\begin{rem}
\label{hypskewsemigrrem}
Let $\G $ be a compact subset of Rat and let 
$f:\GN \times \CCI \rightarrow \GN \times \CCI $ be the skew product 
associated with $\G .$  Let $G$ be the rational semigroup generated by 
$\G . $ Then, by \cite[Remark 2.12]{SdpbpIII}, 
%Lemma~\ref{fiblem}-\ref{pic}, 
%it is easy to see that 
$f$ is semi-hyperbolic if and only if 
$G$ is semi-hyperbolic. Similarly, 
%it is easy to see that 
$f$ is hyperbolic if and only if $G$ is hyperbolic. 
\end{rem}
For a point $z\in \CC $ and a number $r>0$, 
we set $D(z,r):=\{ z\in \CC \mid |y-z|<r\} .$ 
\begin{lem}
\label{nclimlem}
Let $f:X\times \CCI \rightarrow X\times \CCI $ be a 
polynomial skew product over $g:X\rightarrow X$ such that 
for each $\omega \in X$, we have $d(\omega )\geq 2.$ 
Let $x\in X$ and  
 $y_{0}\in F_{x}(f)$.  
 Suppose that there exists a strictly increasing sequence 
 $\{ n_{j}\} _{j\in \NN }$ of positive integers such that   
the sequence $\{ f_{x,n_{j}}\} _{j\in \NN }$  
converges to a non-constant map around $y_{0}$, and such that 
$\lim _{j\rightarrow \infty }f^{n_{j}}(x,y_{0})$ exists. 
We set $(x_{\infty }, y_{\infty }):=
\lim _{j\rightarrow \infty }f^{n_{j}}(x,y_{0}).$ 
Then, there exists a non-empty bounded open set $V$ in $\CC $
and 
a number $k\in \NN$ such that 
$\{ x_{\infty }\} \times \partial V\subset \tilde{J}(f)\cap UH(f)
\subset \tilde{J}(f)\cap P(f)$, 
and such that for each $j$ with $j\geq k$, 
$f_{x,n_{j}}(y_{0})\in V.$  
\end{lem}
\begin{proof}
%By the proof of Lemma 2.13 in \cite{S1}, we obtain the lemma.
%We may assume that the limit 
%$\lim\limits _{n\rightarrow \infty }f^{n_{j}}(x,y_{0})$ 
%exists. We set $(x_{\infty },y_{\infty }):=
%\lim\limits _{n\rightarrow \infty }f^{n_{j}}(x,y_{0}).$ 
%Moreover, 
We set 
$$V:= \{ y\in \CCI \mid 
\exists \epsilon >0, \lim_{i\rightarrow \infty }
\sup _{j>i}\sup _{d(\xi ,y)\leq \epsilon }
d(f_{g^{n_{i}}(x),n_{j}-n_{i}}(\xi ), \xi )=0\} .$$
Then, by \cite[Lemma 2.13]{S1}, 
 we have $\{ x_{\infty }\} \times \partial V
 \subset \tilde{J}(f)\cap UH(f)\subset 
 \tilde{J}(f)\cap P(f).$ 
 Moreover, since for each $x\in X$, 
 $f_{x,1}$ is a polynomial with $d(x)\geq 2$, 
%Lemma~\ref{fibfundlem}-\ref{fibfundlem4} 
\cite[Lemma 3.4(4)]{SdpbpIII} implies that 
there exists a ball $B$ around $\infty $ such that 
$B\subset \CCI \setminus V.$  

 From the assumption, 
 there exists a number $a>0$ and a non-constant map 
 $\varphi :D(y_{0},a)\rightarrow \CCI $ 
 such that $f_{x,n_{j}}\rightarrow \varphi $ 
 as $j\rightarrow \infty $,
 uniformly 
 on $D(y_{0},a).$   Hence,  
  $d(f_{x,n_{j}}(y),f_{x,n_{i}}(y))\rightarrow 0 $ as 
 $i,j\rightarrow \infty $, uniformly on $D(y_{0},a).$ 
 Moreover, since $\varphi $ is not constant, 
% we may assume that 
 there exists a positive number 
 $\epsilon $ such that,  
 for each large $i$, $f_{x,n_{i}}(D(y_{0},a))\supset D(y_{\infty },\epsilon ).$ 
Therefore, it follows that 
$d(f_{g^{n_{i}}(x),n_{j}-n_{i}}(\xi ),\xi )\rightarrow 0$ as $i,j\rightarrow 
\infty $ uniformly on $D(y_{\infty },\epsilon ).$ 
Thus, $y_{\infty }\in V.$ Hence,  
there exists a number $k\in \NN $ such that,  
for each $j\geq k$, $f_{x,n_{j}}(y_{0})\in V.$ 
Therefore, we have proved Lemma~\ref{nclimlem}.
\end{proof}
\begin{rem}
In \cite[Lemma 2.13]{S1} and \cite[Theorem 2.6]{S4}, 
the sequence $(n_{j})$ of positive integers should be strictly 
increasing.
\end{rem}
\begin{lem}
\label{ncintk}
Let $\G $ be a non-empty compact subset of 
{\em Poly}$_{\deg \geq 2}.$ 
Let $f:\GNCR $ be the skew product 
associated with $\G .$ 
Let $G$ be the polynomial semigroup generated by 
$\G .$ 
Let $\g \in \GN $.  
Let $y_{0}\in F_{\g }(f)$ and 
suppose that there exists a strictly increasing sequence 
$\{ n_{j}\} _{j\in \NN }$ of positive integers such that  
$\{ f_{\g ,n_{j}}\} _{j\in \NN }$ converges 
to a non-constant map around $y_{0}.$ 
Moreover, suppose that $G\in {\cal G}.$ 
Then, there exists a number $j\in \NN $ such that 
$f_{\g ,n_{j}}(y_{0})\in $ {\em int}$(\hat{K}(G)).$ 
\end{lem}
\begin{proof}
By Lemma~\ref{nclimlem}, 
there exists a bounded open set $V$ in $\CC $, 
a point $\g _{\infty }\in \GN $, and a number 
$j\in \NN $ such that 
$\{ \g _{\infty }\} \times \partial V\subset 
\tilde{J}(f)\cap P(f)$, and such that 
$f_{\g ,n_{j}}(y_{0})\in V.$  
Then, we have $\partial V\subset P^{\ast }(G).$ 
Since $g(P^{\ast }(G))\subset P^{\ast }(G)$ for each $g\in G$, 
the maximum principle implies that 
$V\subset $ int$(\hat{K}(G)).$ 
Hence, $f_{\g ,n_{j}}(y_{0})\in $ int$(\hat{K}(G)).$ 
Therefore, we have proved Lemma~\ref{ncintk}.
\end{proof}
We now demonstrate statements \ref{mainth3-0} and \ref{mainth3-1} of Theorem~\ref{mainth3}. \\ 
{\bf Proof of statements \ref{mainth3-0} and \ref{mainth3-1} of Theorem~\ref{mainth3}:}
%Let $\g \in R_{S}.$ 
First, we will show the following claim.\\ 
Claim 1. Let $\g \in R(\G ,\G \setminus \G _{\min }).$ Then, 
for any point $y_{0}\in F_{\g }(f)$, there exists 
no non-constant limit function of $\{ f_{\g ,n}\} _{n\in \NN }$ 
around $y_{0}.$ 

 To show this claim, 
%by Proposition~\ref{nonminnoncpt}-\ref{nonminnoncpt3}, 
%we may assume that $\G \setminus \G _{\min }$ is compact.  
 suppose that there exists 
 a strictly increasing sequence $\{ n_{j}\} _{j\in \NN }$ of 
 positive integers
 %subsequence $\{ f_{\g ,n_{j}}\} _{j\in \NN }$ 
% of $\{ f_{\g ,n}\} _{n\in \NN }$ 
such that 
 $f_{\g ,n_{j}}$ tends to a non-constant map as $j\rightarrow \infty $ 
 around $y_{0}.$ 
We consider the following two cases: Case (i): $\G \setminus \G _{\min }$ is compact. Case (ii): $\G \setminus \G _{\min }$ is 
not compact. Suppose that we have Case (i). 
Since $G\in {\cal G}$, Lemma~\ref{ncintk} implies that there exists a 
number $k\in \NN $ such that 
$f_{\g ,n_{k}}(y_{0})\in $ int$(\hat{K}(G)).$ 
% Then, by Lemma~\ref{nclimlem}, 
%there exists a domain $V$ in $\CC $, a point $\alpha $ in 
%$\GN $
%, and a number $k\in \NN $ such that 
%$ \{ \alpha  \} \times \partial V\subset 
%\tilde{J}(f)\cap P(f)$ and 
%$\varphi _{k}(y_{0})\in V.$ 
%Then, we have 
%$\partial V\subset J(G)\cap P^{\ast }(G).$ 
%Since $g(P^{\ast }(G))\subset P^{\ast }(G)$ for each 
%$g\in G$ and $P^{\ast }(G)$ is bounded, 
%by the maximum value principle we obtain that 
%$V\subset $ int$(\hat{K}(G)).$ 
Hence, we get that the sequence 
$\{ f_{\sigma ^{n_{k}}(\g ), n_{k+j}-n_{k}}\} _{j\in \NN }$ converges to a non-constant map around 
the point $y_{1}:=f_{\g ,n_{k}}(y_{0})\in $ int$(\hat{K}(G)).$ 
However, since we are assuming that $\G \setminus \G _{\min }$ is 
compact, 
%Theorem~\ref{mainth2}-\ref{mainth2-4-2} 
\cite[Theorem 2.20.5(b)]{SdpbpI}
implies that  
%we have 
%that for each $h\in \G \setminus \G _{\min }$, 
%$h(\hat{K}(G))\subset $ int$(\hat{K}(G))$, 
$\cup _{h\in \G \setminus \G _{\min }}h(\hat{K}(G))$ is a compact 
subset of int$(\hat{K}(G))$, 
which implies 
that if we take the hyperbolic metric for each connected component 
of int$(\hat{K}(G))$, then there exists a constant $0<c<1$ such 
that for each $z\in $ int$(\hat{K}(G))$ and 
each $h\in \G \setminus \G _{\min }$, we have 
$\| h'(z)\| \leq c$, where $\| h'(z) \| $ denotes 
the norm of the derivative of $h$ at $z$ measured from the 
hyperbolic metric on the connected component $W_{1}$ of int$(\hat{K}(G))$ 
containing $z$ to that of the connected component $W_{2}$ of int$(\hat{K}(G))$ 
containing $h(z).$ This leads to a contradiction, since 
we have that 
$\g \in R(\G ,\G \setminus \G _{\min })$ 
%$x\in \Sigma _{m}\setminus \cup _{n\geq 0}\sigma ^{-n}
%(B_{\min }^{\NN })$ and 
and the sequence 
$\{ f_{\sigma ^{n_{k}}(\g ), n_{k+j}-n_{k}}\} _{j\in \NN }$ converges to a non-constant map around 
the point $y_{1}\in $ int$(\hat{K}(G)).$ 
We now suppose that we have Case (ii). 
Then, combining the arguments in Case (i) and \cite[Theorem 2.20.5(b), Proposition 2.33]{SdpbpI}, 
we again obtain a contradiction. 
Hence, we have shown Claim 1.

 Next, let $S$ be a non-empty 
 compact subset of $\G \setminus \G _{\min }$ and let 
 $\g \in R(\G ,S).$  
 We show the following claim.\\ 
Claim 2. For each point $y_{0}$ in each bounded component of $F_{\g }(f)$,  
%(under the natural identification $\pi ^{-1}\{ \g \} \cong \CCI $), 
there exists a number $n\in \NN $ such that 
$f_{\g ,n}(y_{0})\in $ int$(\hat{K}(G)).$ 

 To show this claim, we suppose that there exists no $n\in \NN $ 
 such that $f_{\g ,n}(y_{0})\in $ int$(\hat{K}(G))$.  
%and 
% we  deduce a contradiction. 
By Claim 1, $\{ f_{\g ,n}\} _{n\in \NN }$ has only constant limit functions 
around $y_{0}.$ Moreover, if a point $w_{0}\in \CC $ 
is a constant limit function of 
$\{ f_{\g ,n}\} _{n\in \NN }$, then since $G\in {\cal G}$, 
\cite[Lemma 3.13]{SdpbpIII} implies that 
%Lemma~\ref{constlimlem}, 
we must have that $w_{0}\in P^{\ast }(G)\subset \hat{K}(G).$ Since 
we are assuming that there exists no $n\in \NN $ 
 such that $f_{\g ,n}(y_{0})\in $ int$(\hat{K}(G))$, 
 it follows that $w_{0}\in \partial \hat{K}(G).$ 
Combining this with \cite[Theorem 2.20.2]{SdpbpI}
%Theorem~\ref{mainth2}-\ref{mainth2-3}, 
we deduce that $w_{0}\in \partial \hat{K}(G)\subset J_{\min }.$  
% Hence, we have that $w_{0}$ must belong to the set 
% $P^{\ast }(G)\cap J(G)$, which is included in $J_{\min }$, 
% by Theorem~\ref{mainth2}-\ref{mainth2-3}. 
From this argument, we get that 
\begin{equation}
\label{main3-1pfeq1}
d(f_{\g ,n}(y_{0}),J_{\min })\rightarrow 0, \mbox{ as }n\rightarrow \infty.
\end{equation}  
%and 
%let $(x,y_{0})$ be a point in $F_{x}(f).$ 
However, since  
$\g$  belongs to $R(\G ,S)$, 
the above (\ref{main3-1pfeq1}) implies that 
the sequence $\{ f_{\g ,n}(y_{0})\} _{n\in \NN }$ accumulates 
in the compact set 
$\cup _{h\in S}
h^{-1}(J_{\min})$, which is apart from 
$J_{\min }$, by \cite[Theorem 2.20.5(b)]{SdpbpI}. 
%Theorem~\ref{mainth2}-\ref{mainth2-4-2}. 
This contradicts (\ref{main3-1pfeq1}). Hence, we have shown that 
Claim 2 holds. 

  Next, we show the following claim.\\ 
Claim 3. There exists exactly one bounded 
component $U_{\g }$ of $F_{\g }(f).$ 
%such that 
%  $\pi _{\CCI }(U_{\g })$ is a bounded component 
%  of $ \pi _{\CCI}(F_{\g }(f))\cap \CC .$ 

 To show this claim, we take an element $h\in \G_{\min }$ 
(note that $\G_{\min }\neq \emptyset$, by Proposition~\ref{bminprop}).
We write the element $\g $ as $\g =(\g _{1},\g_{2},\ldots ).$ 
For any $l\in \NN $ with $l\geq 2$, 
let $s_{l}\in \NN $ be an integer 
with $s_{l}>l$ such that  
$\g_{s_{l}}\in S.$ 
We may assume that for each $l\in \NN $, 
$s_{l}<s_{l+1}.$ For each $l\in \NN $, 
let $\g ^{l}:= (\g _{1},\g _{2},\ldots ,\g _{s_{l}-1},h,h,h,\ldots )\in 
\GN $ and  
$\tilde{\g }^{l}:=\sigma ^{s_{l}-1}(\g )
=(\g _{s_{l}}, \g _{s_{l}+1},\ldots )
\in \GN .$ Moreover,  
let $\rho :=(h,h,h,\ldots )\in \GN .$ 
Since $h\in \G _{\min }$, 
%we have 
%that $h^{-1}(J_{\min })\subset J_{\min }.$ 
%Combined with $\sharp (J_{\min })\geq 3$ 
%(Theorem~\ref{mainth2}-\ref{mainth2-4}) and Lemma~\ref{hmslem}-\ref{backmin}, 
we have  
\begin{equation}
\label{main3-1pfeq2}
J_{\rho }(f)=J(h)\subset 
J_{\min }.
\end{equation} 
Moreover, since $\g _{s_{l}}$ does not belong to 
$\G _{\min }$, combining it with \cite[Theorem 2.20.5(b)]{SdpbpI}, 
%Theorem~\ref{mainth2}-\ref{mainth2-4-2} 
%Theorem~\ref{mainth1}-\ref{mainth1-3}, 
we obtain  
$\g_{s_{l}}^{-1}(J(G))\cap J_{\min }=\emptyset .$ 
Hence, we have that for each $l\in \NN $, 
\begin{equation}
\label{main3-1pfeq3}
J_{\tilde{\g }^{l}}(f)=
\g _{s_{l}}^{-1}(J_{\sigma ^{s_{l}}(\g )}(f))
\subset  
\g_{s_{l}}^{-1}(J(G)) \subset \CCI \setminus J_{\min }.
\end{equation} 
Combining (\ref{main3-1pfeq2}),  (\ref{main3-1pfeq3}), 
and \cite[Lemma 3.9]{SdpbpIII}
%Lemma~\ref{fiborder}, 
we obtain 
\begin{equation}
\label{main3-1pfeq3a}
J_{\rho }(f)<J_{\tilde{\g }^{l}}(f), 
\end{equation}
which implies   
\begin{equation}
\label{main3-1pfeq4}
J_{\g ^{l}}(f)=
(f_{\g , s_{l}-1})^{-1}(J_{\rho }(f))
<(f_{\g ,s_{l}-1})^{-1}(J_{\tilde{\g }^{l}}(f))
=J_{\g }(f).
\end{equation} 
%Since the union of bounded components of 
%$\pi _{\CCI }(F_{x^{l}}(f))\cap \CC $ 
%contains the set int $\hat{K}(G)$, which is non-empty, 
%it follows that there exists a positive constant $C$ such that 
%for each $l\in \NN $ with $l\geq 2$, 
%diam $\pi _{\CCI }(J_{x^{l}}(f))\geq C.$
From \cite[Lemma 3.9]{SdpbpIII} 
%Lemma~\ref{fiborder} 
and 
(\ref{main3-1pfeq4}), 
it follows that 
%we may assume that 
there exists a bounded component 
$U_{\g }$ of $F_{\g }(f)$ 
%(under the natural identification 
%$\pi ^{-1}\{ \g \} \cong \CCI $) 
such that 
for each $l\in \NN $ with $l\geq 2$, 
\begin{equation}
\label{main3-1pfeq5}
J_{\g ^{l}}(f)\subset U_{\g }.
\end{equation}
We now suppose that there exists a bounded component $V$ of 
$F_{\g }(f)$ with $V\neq U_{\g }$, and we will deduce a contradiction.
Under the above assumption, we take a point $y\in V.$ 
Then, by Claim 2, we get that 
there exists a number $l\in \NN $ such that 
$f_{\g ,l}(y)\in $ int$(\hat{K}(G)).$ 
Since $s_{l}> l$, we obtain 
$f_{\g ,s_{l}-1}(y)\in $ int$(\hat{K}(G))\subset K(h)$, 
where, $h\in \G _{\min }$ is the element which we have taken before.  
%Since $\pi _{\CCI }(J_{j^{\infty }}(f)<
% \pi _{\CCI }(J_{x^{l}}(f))$, which can be shown by the same 
% method as that of the proof of 
%Since $K(h_{j})=f_{x^{l}, s_{l}-1}
%\left(\pi _{\CCI }(K_{x^{l}}(f))\right) $, 
%combined with 
By (\ref{main3-1pfeq3a}), 
we have that there exists a bounded component 
$B$ of $F_{\tilde{\g }^{l}}(f) $ 
containing $K(h).$ 
Hence, we have 
%$q_{x^{l}}^{(s_{l}-1)}(y)=
$f_{\g ,s_{l}-1}(y)\in B.$ 
Since the map $f_{\g ,s_{l}-1}:V\rightarrow B$ 
is surjective, it follows that 
$V\cap \left((f_{\g ,s_{l}-1})^{-1}(J(h))\right) 
\neq \emptyset .$ 
Combining this  with $ (f_{\g ,s_{l}-1})^{-1}(J(h))=
(f_{\g ^{l}, s_{l}-1})^{-1}(J(h))=
J_{\g ^{l}}(f)$, 
we obtain 
$V\cap J_{\g ^{l}}(f)   
 \neq \emptyset .$ 
 However, this leads to a contradiction, 
since we have (\ref{main3-1pfeq5}) and $U_{\g }\cap V=\emptyset .$ 
Hence, we have shown Claim 3.

 Next, we show the following claim. \\ 
Claim 4. We have 
$\partial U_{\g }=\partial A_{\gamma }(f)= J_{\g }(f).$ 

 To show this claim, 
% for each $l\in \NN $, 
%  we take the element $\g ^{l}\in \GN $ 
% as before. 
%Then, we have $\g ^{l}\rightarrow \g $ in 
%$\GN $, as $l\rightarrow \infty .$ 
%Let $w\in J_{\g }(f)$ be any point. 
%Then, by Lemma~\ref{fibfundlem}-\ref{fibfundlem2}, 
%there exists a sequence $(w_{l})$ in $\GN \times \CCI 
%$ such that $w_{l}\in J_{\g ^{l}}(f)$ and $w_{l}\rightarrow 
%w.$ By (\ref{main3-1pfeq5}), we have 
%$\pi _{\CCI }(w_{l})\in \pi _{\CCI }(U_{\g }).$ 
%Hence, it follows that 
%$\pi _{\CCI }(w)\in \partial (\pi _{\CCI }(U_{\g }))$, 
%which implies 
%$J_{\g }(f)\subset \partial U_{\g }.$ Hence,  
%we have that $J_{\g }(f)=\partial U_{\g }$ in $\pi ^{-1}\{ \g \} .$ 
since $U_{\g }=\mbox{int}(K_{\g }(f))$, 
%in $\pi ^{-1}\{ \g \} $, 
%Lemma~\ref{fibfundlem}-\ref{fibfundlema} 
\cite[Lemma 3.4(5)]{SdpbpIII}
implies that 
$\partial U_{\g }=J_{\g}(f).$ Moreover, by 
%Lemma~\ref{fibfundlem}-\ref{fibfundlem4}, 
\cite[Lemma 3.4(4)]{SdpbpIII}
we have 
$\partial A_{\gamma }(f)=J_{\gamma }(f).$  
Thus, we have shown Claim 4.

 We now show the following claim.\\ 
Claim 5. We have $\hat{J}_{\g }(f)=J_{\g }(f)$ and 
the map $\omega \mapsto J_{\omega }(f)$ is continuous at $\g $
 with respect to the Hausdorff metric in the space of non-empty 
 compact subsets of $\CCI .$  

 To show this claim, 
 suppose that there exists a point $z$ with 
 $z\in \hat{J}_{\g }(f)\setminus J_{\g }(f).$ 
Since $\hat{J}_{\g }(f)\setminus 
J_{\g }(f)$ is included in the union 
of bounded components of $F_{\g }(f)$, 
%(under the natural identification $\pi ^{-1}\{ \g \} \cong \CCI $), 
combining it with Claim 2,   
we get that there exists a number $n\in \NN $ such that 
$f_{\g ,n}(z)\in $ int$(\hat{K}(G))\subset F(G).$ 
However, since 
$z\in \hat{J}_{\g }(f)$, we must have 
that $f_{\g ,n}(z)=\pi _{\CCI }(f_{\g }^{n}(z))\in 
\pi _{\CCI }(\tilde{J}(f))=J(G).$ This is a contradiction. 
Hence, we obtain $\hat{J}_{\g }(f)=J_{\g }(f).$ 
Combining this with 
\cite[Lemma 3.4(2)]{SdpbpIII}, 
%Lemma~\ref{fibfundlem}-\ref{fibfundlem2}, 
it follows that $\omega \mapsto J_{\omega }(f)$ is 
continuous at $\g .$   
Therefore, we have shown Claim 5. 

 Combining all Claims $1,\ldots ,5$, it follows   
 that statements \ref{mainth3-0}, 2(a), 2(b) and 2(c) 
%\ref{mainth3-1-1}, \ref{mainth3-1-2}, 
% and \ref{mainth3-1-3} in 
of Theorem~\ref{mainth3}
 hold.  
%By Lemma~\ref{nclimlem} and 
%$h_{j}(\hat{K}(G))\subset $ int$(\hat{K}(G))$ for each 
%$j\in \{ 1,\cdots ,m\} \setminus B_{\min }$  
%(Theorem~\ref{mainth2}-\ref{mainth2-4-2}), 
%we obtain that there exists no non-constant limit 
%function of the sequence 
%$(q_{x}^{(n)})_{n\in \NN }$ 
%(q_{x}^{(n)}(y):= \pi _{\CCI }f_{x}^{n}((x,y)))$ 
%around 
%$y_{0}.$ Since $P^{\ast }(G)\cap J(G)
%\subset J_{\min }$ (Theorem~\ref{mainth2}-\ref{mainth2-3}), 
%we obtain that statement \ref{mainth3-1-2} is true. 
%From statement \ref{mainth3-1-2}, we obtain the 
%statement \ref{mainth3-1-3}. 
%By the lower semi-continuity of 
%$x\mapsto J_{x}(f)$ (Lemma~\ref{fibfundlem}-\ref{fibfundlem2}), 
%we obtain statement \ref{mainth3-1-1}.
%By Theorem~\ref{mainth1}, we obtain 
%$h_{j}^{-1}(J(G))\cap J_{\min }=\emptyset $ for each 
%$j\in \{ 1,\cdots ,m\} \setminus 
%B_{\min }.$  
%Combining this with Theorem~\ref{mainth2}-\ref{mainth2-3} and 
%the Koebe distortion theorem, 
%we obtain statement \ref{mainth3-1-4}. 

 We now show statement 2(d). 
%\ref{mainth3-1-4}. 
Let 
 $\g \in R(\G ,S)$ be an element. 
Suppose that $m_{2}(J_{\g }(f))>0$, where $m_{2}$ denotes the 
$2$-dimensional Lebesgue measure. Then, there exists a 
Lebesgue density point $b\in J_{\g }(f)$ so that 
\begin{equation}
\label{lebpteq}
\lim _{s\rightarrow 0}
\frac{m_{2}\left(D(b,s)\cap J_{\g }(f)\right)}
{m_{2}(D(b,s))}=1.
\end{equation} 
Since $\g $ belongs to $R(\G ,S)$, 
%$\cup _{n\in \NN }\sigma ^{-n}(B_{\min }^{\NN })$, 
there exists an element $\g _{\infty }\in S$ and  
a sequence $\{ n_{j}\} _{j\in \NN }$ of positive integers  
such that $n_{j}\rightarrow \infty $ and 
$\g _{n_{j}}\rightarrow \g _{\infty }$ 
as $j\rightarrow \infty $,  
and such that for each $j\in \NN $, we have $\g _{n_{j}}\in S.$ 
We set $b_{j}:= f_{\g ,n_{j}-1}(b)$, for each $j\in \NN .$ 
We may assume that there exists a point $a\in \CC $ 
such that $b_{j}\rightarrow a$ as $j\rightarrow \infty .$ 
Since $\g _{n_{j}}(b_{j})=f_{\g ,n_{j}}(b)=
\pi _{\CCI }(f_{\g }^{n_{j}}(\g ,b))\in 
\pi _{\CCI }(\tilde{J}(f))=J(G)$, 
we obtain  
%$b_{j}\in h_{i}^{-1}(J(G))$, for each $j\in \NN $, which 
%implies $a\in h_{i}^{-1}(J(G)).$ 
$a\in \g _{\infty }^{-1}(J(G)).$ Moreover, 
by \cite[Theorem 2.20.5(b)]{SdpbpI}, 
%Theorem~\ref{mainth2}-\ref{mainth2-4-2}, 
we obtain 
\begin{equation}
\label{mainth314pfeq0.5}
a\in \g _{\infty }^{-1}(J(G))\subset \CC \setminus J_{\min }.
\end{equation}
Combining this with \cite[Theorem 2.20.2]{SdpbpI}, 
%Theorem ~\ref{mainth2}-\ref{mainth2-3}, 
%and 
%Theorem~\ref{mainth2}-\ref{mainth2-4-2}, 
it follows that 
%\begin{equation}
%\label{mainth314pfeq1}
$r:=\inf \{ |a-b| \mid b\in P^{\ast }(G)\} >0.$
%\end{equation}
Let $\epsilon $ be arbitrary number with $0<\epsilon <\frac{r}{10}.$
We may assume that for each $j\in \NN $, 
we have $b_{j}\in D(a,\frac{\epsilon }{2}).$  
 For each $j\in \NN ,$ let $\varphi _{j}$ be the well-defined inverse branch 
of $f_{\g ,n_{j}-1}$ on 
$D(a,r)$ such that $\varphi _{j}(b_{j})=b.$  
Let $V_{j}:= \varphi _{j}(D(b_{j},r-\epsilon ))$, for each 
$j\in \NN .$  
We now show the following claim.\\ 
Claim 6. diam $V_{j}\rightarrow 0$, 
as $\j\rightarrow \infty .$ 

 To show this claim, suppose that this is not true. Then, there exists a 
 strictly increasing sequence $\{ j_{k}\} _{k\in \NN }$ of  
 positive integers and a positive constant $\kappa $ 
 such that for each $k\in \NN $, diam $V_{j_{k}}\geq \kappa .$ 
 From the Koebe distortion theorem, it follows that 
there exists a positive constant $c_{0}$ such that for each 
$k\in \NN $, 
$V_{j_{k}}\supset D(b,c_{0}).$ This implies that for each $k\in \NN $, 
$f_{\g ,v_{k}}(D(b,c_{0}))\subset 
D(b_{j_{k}},r-\epsilon )$, where $v_{k}:= n_{j_{k}}-1.$ 
Since $v_{k}\rightarrow \infty $ as $k\rightarrow \infty $ and 
$f_{\g ', n}|_{F_{\infty }(G)}\rightarrow \infty $ for any 
$\g '\in \GN $, 
%where $F_{\infty }(G)$ denotes the 
%component of $F(G)$ containing $\infty $, 
it follows that 
for any $n\in \NN $, 
$f_{\g ,n}(D(b,c_{0}))\subset \CCI \setminus F_{\infty }(G)$,  
which implies that $b\in F_{\g }(f).$ However,  
it contradicts $b\in J_{\g }(f).$ 
Hence, Claim 6 holds. 

 Combining the Koebe distortion theorem and Claim 6, 
 we see that 
 there exist a constant $K>0$ and two sequences $\{ r_{j}\} _{j\in \NN }$ and 
 $\{ R_{j}\} _{j\in \NN }$ of positive numbers such that 
 $K\leq \frac{r_{j}}{R_{j}}<1$ and 
 $D(b,r_{j})\subset V_{j}\subset D(b,R_{j})$ for each $j\in \NN $, 
 and such that $R_{j}\rightarrow 0$ as $j\rightarrow \infty $. 
From (\ref{lebpteq}), it follows that 
\begin{equation}
\label{fatou0eq1}
\lim _{j\rightarrow \infty }
\frac{m_{2}\left(V_{j}\cap F_{\g }(f)\right)}
{m_{2}(V_{j})}=0.
\end{equation}   
For each $j\in \NN $, let 
$\psi _{j}: D(0,1)\rightarrow \varphi _{j}(D(a,r))$ be a biholomorphic 
map such that 
$\psi _{j}(0)=b.$ 
Then, there exists a constant $0<c_{1}<1$ such that 
for each $j\in \NN $, 
\begin{equation}
\label{psid0}
\psi _{j}^{-1}(V_{j})\subset 
D(0,c_{1}).
\end{equation}
Combining this with (\ref{fatou0eq1}) and the Koebe distortion theorem, 
it follows that  
\begin{equation}
\label{fatou0eq2}
\lim _{j\rightarrow \infty }
\frac{m_{2}\left(\psi _{j}^{-1}(V_{j}\cap F_{\g }(f))\right)}
{m_{2}(\psi _{j}^{-1}(V_{j}))}=0.
\end{equation}   
Since $\varphi _{j}^{-1}(\psi _{j}(D(0,1)))\subset 
D(a,r)$ for each $j\in \NN $, 
combining (\ref{psid0}) and 
Cauchy's formula yields that 
there exists a constant $c_{2}>0$ such that for any $j\in \NN $,  
\begin{equation}
\label{fatou0eq3}
|(f_{\g ,n_{j}-1}\circ \psi _{j})'(z)|\leq c_{2}\ \mbox{ on }
\psi _{j}^{-1}(V_{j}).
\end{equation}     
Combining (\ref{fatou0eq2}) and (\ref{fatou0eq3}), 
we obtain 
\begin{align*}
 \frac{m_{2}\left(D(b_{j},r-\epsilon )\cap 
F_{\sigma ^{n_{j}-1}(\g )}(f)\right)}
{m_{2}(D(b_{j},r-\epsilon ))} 
 =  \frac{m_{2}\left((f_{\g ,n_{j}-1}\circ \psi _{j})
(\psi _{j}^{-1}(V_{j}\cap 
F_{\g }(f)))\right)}
{m_{2}(D(b_{j},r-\epsilon ))}\\ 
 = 
\frac{\int _{\psi _{j}^{-1}(V_{j}\cap F_{\g }(f))}
|(f_{\g ,n_{j}-1}\circ \psi _{j})'(z)|^{2}\ dm_{2}(z)}
{m_{2}(\psi _{j}^{-1}(V_{j}))}
\cdot 
\frac{m_{2}(\psi _{j}^{-1}(V_{j}))}
{m_{2}(D(b_{j},r-\epsilon ))} 
 \rightarrow 0,
\end{align*}
as $j\rightarrow \infty .$ 
Hence, we obtain 
$$\lim _{j\rightarrow \infty }
\frac{m_{2}\left(D(b_{j},r-\epsilon )\cap 
J_{\sigma ^{n_{j}-1}(\g )}(f)\right)}
{m_{2}(D(b_{j},r-\epsilon ))}=1.$$
Since  $J_{\sigma ^{n_{j}-1}(\g )}(f)\subset J(G)$ 
for each $j\in \NN $, and $b_{j}\rightarrow a$ as $j\rightarrow \infty $, 
it follows that 
$$\frac{m_{2}(D(a,r-\epsilon )\cap J(G))}
{m_{2}(D(a,r-\epsilon ))}=1.$$ 
This implies that 
$D(a,r-\epsilon )\subset J(G).$ 
Since this is valid for any $\epsilon $, 
we must have that $D(a,r)\subset J(G).$ It follows that  
the point $a$ belongs to a connected component 
$J$ of $J(G)$ such that $J\cap P^{\ast }(G)\neq \emptyset .$ 
However, \cite[Theorem 2.20.2]{SdpbpI} 
%Theorem~\ref{mainth2}-\ref{mainth2-3} 
implies 
that the component $J$ is equal to $J_{\min }$, 
which leads to a contradiction since we have  
% $a\in h_{i}^{-1}(J(G))\subset \CCI \setminus J_{\min }$ by 
%Theorem~\ref{mainth2}-\ref{mainth2-4-2}.  
(\ref{mainth314pfeq0.5}).
 Hence, we have shown statement 2(d) 
%\ref{mainth3-1-4} 
of 
 Theorem~\ref{mainth3}. 

 Therefore, we have proved statements \ref{mainth3-0} 
 and \ref{mainth3-1} of Theorem~\ref{mainth3}.  
\qed 
%\simeq 

We now demonstrate statement \ref{mainth3-3} of Theorem~\ref{mainth3}. \\ 
\noindent 
{\bf Proof of statement \ref{mainth3-3} of Theorem~\ref{mainth3}:}
First, we remark that the subset $W_{S,p}$ of 
$\GN $ is a $\sigma $-invariant compact set.
 Hence, $\overline{f}:
W_{S,p}\times \CCI \rightarrow W_{S,p}\times \CCI $ 
is a polynomial skew product over 
$\sigma :W_{S,p}\rightarrow W_{S,p}.$  
Suppose that $\tilde{J}(\overline{f})\cap 
P(\overline{f})\neq \emptyset $ 
and let $(\g ,y)\in \tilde{J}(\overline{f})\cap 
P(\overline{f})$ be a point. 
Then, since the point $\g =(\g _{1},\g _{2},\ldots )$ belongs to 
$W_{S,p}$, 
there exists a number $j\in \NN $ such that 
$\g _{j}\in S.$ 
Combining this with the condition that $G\in {\cal G}_{dis}$ and  
\cite[Theorem 2.20.5(b), Theorem 2.20.2]{SdpbpI}, 
%Theorem~\ref{mainth2}-\ref{mainth2-4-2} and 
%Theorem~\ref{mainth2}-\ref{mainth2-3}, 
we have $\g_{j}^{-1}(J(G))\subset \CC \setminus 
\hat{K}(G) \subset \CC \setminus P(G).$ 
Moreover, we have that 
$\pi _{\CCI }(\overline{f}_{\g }^{j-1}(\g ,y))
= \pi _{\CCI }(f_{\g }^{j-1}(\g ,y))
\in J_{\sigma ^{j-1}(\g )}(f)
= \g _{j}^{-1}\left(J_{\sigma ^{j}(\g )}(f)\right)
\subset \g_{j}^{-1}(J(G)).$ 
Hence, we obtain 
\begin{equation}
\label{mainth3-3pfeq1}
\pi _{\CCI }(\overline{f}_{\g }^{j-1}(\g ,y))\in \CC \setminus 
P(G).
\end{equation} 
However, since $(\g ,y)\in P(\overline{f})$, 
we have that $\pi _{\CCI }(\overline{f}_{\g }^{j-1}(\g ,y))
\in \pi _{\CCI }(P(\overline{f}))
\subset P(G)$, which contradicts (\ref{mainth3-3pfeq1}). 
Hence, we must have that 
$\tilde{J}(\overline{f})\cap P(\overline{f})=
\emptyset .$ Therefore, 
$\overline{f}:W_{S,p}\times \CCI \rightarrow W_{S,p}\times \CCI $ 
is a hyperbolic polynomial skew product over 
the shift map $\sigma :W_{S,p}\rightarrow W_{S,p}.$

%Since $P^{\ast }(G) \cap J(G)
%\subset J_{\min }$ (Theorem~\ref{mainth2}-\ref{mainth2-3}), 
%it is easy to see 
%$\hat{J}_{x}(f)\cap P(f)=\emptyset $
%$(\pi _{\CCI }\hat{J}_{x}(f))\cap P(G)=\emptyset $ 
%for 
%each $x\in W_{s}.$ Hence, $\overline{f}$ is a hyperbolic 
%skew product. 
Combining this with statement 2(a) of Theorem~\ref{mainth3} and 
%Theorem~\ref{hypskewqc}, 
\cite[Theorem 4.1]{SdpbpIII}
we conclude that 
there exists a constant $K_{S,p}\geq 1$ such that 
for each $\g \in W_{S,p},\ J_{\g }(\overline{f})$ is 
a $K_{S,p}$-quasicircle.
Moreover, by statement 2(c) of Theorem~\ref{mainth3},  
we have $J_{\g }(\overline{f})=J_{\g }(f)=\hat{J}_{\g }(f).$ 

Hence, we have shown statement \ref{mainth3-3} of Theorem~\ref{mainth3}.  
\qed \\ 
 We now demonstrate Theorem~\ref{mainth3-2}.\\ 
\noindent {\bf Proof of Theorem~\ref{mainth3-2}:}
Let 
$\g \in R(\G ,\G \setminus \G _{\min })$ and 
$y\in $ int$(K_{\g }(f)).$ 
Combining statement \ref{mainth3-0} of Theorem~\ref{mainth3}  
and \cite[Lemma 1.10]{S1}, 
we obtain $\liminf _{n\rightarrow \infty }$ $
d(f_{\g,n}(y),$ $J(G))>0.$ 
Combining this with 
%Lemma~\ref{constlimlem} 
\cite[Lemma 3.13]{SdpbpIII}
and statement \ref{mainth3-0} of 
Theorem~\ref{mainth3},  
we see that there exists a point 
$a\in P^{\ast }(G)\cap F(G)$ such that 
$\liminf _{n\rightarrow \infty }d(f_{\g ,n}(y),a)=0.$ 
Since $P^{\ast }(G)\cap F(G)\subset $ int$(\hat{K}(G))$ 
(which follows from the condition that $G\in {\cal G}$),  
it follows that there exists a positive integer 
$l$ such that 
%\begin{equation}
%\label{mainth3-2pfeq0}
$f_{\g ,l}(y)\in \mbox{int}(\hat{K}(G)).$
%\end{equation} 
Combining this  
%(\ref{mainth3-2pfeq0}) 
and the same method as that in the proof of Claim 3 in 
the proof of statements \ref{mainth3-0} and \ref{mainth3-1} of Theorem~\ref{mainth3},  
we get that there exists exactly one 
bounded component $U_{\g }$ of $F_{\g }(f).$ 
Combining it with 
\cite[Proposition 4.6]{SdpbpIII}, 
%Proposition~\ref{shonecomp}, 
it follows that $J_{\g }(f)$ is a Jordan curve. 
%
% We now show $\hat{J}_{\g }(f)=J_{\g }(f).$ 
% Suppose that there exists a point 
% $y_{0}\in \hat{J}_{\g }(f)\setminus J_{\g }(f).$ 
% Since $\hat{J}_{\g }(f)\subset K_{\g }(f)$, 
% it follows that $y_{0}\in U_{\g }.$ 
% Combining it with (\ref{mainth3-2pfeq0}), 
% we get that there exists a positive integer 
% $l$ such that $f_{\g ,l}(y_{0})\in $ int$(\hat{K}(G))\subset 
% F(G).$  However, 
% $f_{\g ,l}(y_{0})\in f_{\g ,l}(\hat{J}_{\g }(f))
% \subset \hat{J}_{\sigma ^{l}(\g )}(f)\subset J(G)$, 
% which is a contradiction. Therefore, 
% $\hat{J}_{\g }(f)=J_{\g }(f).$ 
% Combining it with the lower semicontinuity of 
% the map $x\mapsto J_{x}(f)$ (cf. 
% Lemma~\ref{fibfundlem}-\ref{fibfundlem2}), 
% it follows that the map 
% $\omega \mapsto J_{\omega }(f)$ is continuous at $\g .$ 
Moreover, by \cite[Theorem 2.14-(4)]{S1}, 
we have $\hat{J}_{\gamma }(f)=J_{\gamma }(f).$ 

 Thus, we have proved Theorem~\ref{mainth3-2}.
\qed

\  

 We now demonstrate Theorem~\ref{cantorqc}.

\noindent 
{\bf Proof of Theorem~\ref{cantorqc}:}
%Since $J(G)=\overline{\cup _{g\in G}J(g)}$ 
%(Theorem~\ref{repdense}), there exists 
%an element $h_{1}\in G$ with 
%$J(h_{1})\cap V\neq \emptyset .$
%By Theorem~\ref{mainth0},
%there exists an element $h_{2}\in G$ such that 
%the Julia set of $G_{1}=\langle h_{1},h_{2}\rangle $ 
%is disconnected. By Theorem~\ref{mainth3}-\ref{mainth3-3}, 
%we can find two elements $g_{1}$ and $g_{2}$ in $G_{1}$ 
%satisfying all of the conditions in the statement in 
%Theorem~\ref{cantorqc}.
Let $V$ be an open set with $J(G)\cap V\neq \emptyset .$ 
We may assume that $V$ is connected.
Then, by \cite[Corollary 3.1]{HM1}
%Theorem~\ref{repdense}, 
there exists an element $\alpha _{1}\in G$ such that 
$J(\alpha _{1})\cap V\neq \emptyset .$ 
Since we have $G\in {\cal G}_{dis}$,  
%Theorem~\ref{mainth0} 
\cite[Theorem 2.1]{SdpbpI}
implies that there exists an element 
$\alpha _{2}\in G$ such that 
no connected component $J$ of $J(G)$ satisfies  
$J(\alpha _{1})\cup J(\alpha _{2})\subset J.$ 
Hence, we have $\langle \alpha _{1}, \alpha _{2}\rangle \in {\cal G}_{dis}.$ 
Since $J(\alpha _{1})\cap V\neq \emptyset $, 
combining this with 
\cite[Lemma 3.4(2)]{SdpbpIII} 
%Lemma~\ref{fibfundlem}-\ref{fibfundlem2}, 
we get that there exists an $l_{0}\in \NN $ such that 
for each $l$ with $l\geq l_{0}$, we have 
$J(\alpha _{2}\alpha _{1}^{l})\cap V\neq \emptyset .$ 
Moreover, since no connected component $J$ of $J(G)$ satisfies  
$J(\alpha _{1})\cup J(\alpha _{2})\subset J$, 
%Lemma~\ref{fibfundlem}-\ref{fibfundlem2} 
\cite[Lemma 3.4(2)]{SdpbpIII}
implies that 
there exists an $l_{1}\in \NN $ such that 
for each $l$ with $l\geq l_{1}$, 
$J(\alpha _{2}\alpha _{1}^{l})\cap 
J(\alpha _{1}\alpha _{2}^{l})=\emptyset .$ 
We fix an $l\in \NN $ with $l\geq \max \{ l_{0},l_{1}\}.$ 
We now show the following claim.\\ 
\noindent Claim 1. 
The semigroup $H_{0}:= \langle \alpha _{2}\alpha _{1}^{l}, 
\alpha _{1}\alpha _{2}^{l}\rangle $ is hyperbolic, and 
for the skew product $\tilde{f}:\GN _{0}\times \CCI 
\rightarrow \GN _{0}\times \CCI $ associated with 
$\G _{0}=\{ \alpha _{2}\alpha _{1}^{l}, \alpha _{1}\alpha _{2}^{l}\} $, 
there exists a constant $K\geq 1$ such that 
for any $\g \in \GN _{0}$, $J_{\g }(\tilde{f})$ is a $K$-quasicircle. 

 To show this claim, 
applying statement \ref{mainth3-3} of Theorem~\ref{mainth3} 
with $\G =\{ \alpha _{1},\alpha _{2}\} , 
S=\G \setminus \G _{\min }, $ and $p=2l+1$, 
%$m=2,h_{1}=\alpha _{1},h_{2}=\alpha _{2}$, and $s=l+1$), 
we see that 
the polynomial skew product 
$\overline{f}:W_{S, 2l+1}\times \CCI \rightarrow W_{S,2l+1}\times 
\CCI $ over $\sigma :W_{S,2l+1}\rightarrow W_{S,2l+1}$ 
is hyperbolic, and that 
there exists a constant $K\geq 1$ such that 
for each $\g \in W_{S,2l+1}$, $J_{\g }(\overline{f})$ 
is a $K$-quasicircle. 
Moreover, combining the hyperbolicity of $\overline{f}$ above and 
%Lemma 2.10 in \cite{S1}, 
Remark~\ref{hypskewsemigrrem}, 
we see that the semigroup 
$H_{1}$ generated by the family 
$\{ \alpha_{j_{1}}\circ \cdots \circ \alpha _{j_{l+1}}\mid 
1\leq \exists k_{1}\leq l+1 \mbox{ with } j_{k_{1}}=1, 
1\leq \exists k_{2}\leq l+1 \mbox{ with }j_{k_{2}}=2\} $ 
is hyperbolic. 
Hence, the semigroup $H_{0}$, which is a subsemigroup of $H_{1}$, 
is hyperbolic. 
 Therefore, Claim 1 holds.

 We now show the following claim.\\ 
Claim 2. We have either 
$J(\alpha _{2}\alpha _{1}^{l})<J(\alpha _{1}\alpha _{2}^{l})$,  
or $J(\alpha _{1}\alpha _{2}^{l})<J(\alpha _{2}\alpha _{1}^{l}).$

 To show this claim, since 
$J(\alpha _{2}\alpha _{1}^{l})\cap J(\alpha _{1}\alpha _{2}^{l})   
=\emptyset $ and $H_{0}\in {\cal G}$, 
%it is easy to see the claim holds, using the method in the proof of 
%Theorem~\ref{mainth1}-\ref{mainth1-1}.
combining these with 
%Lemma~\ref{fiborder}, 
\cite[Lemma 3.9]{SdpbpIII}, 
we obtain Claim 2. 

 By Claim 2, we have the following two cases.\\ 
Case 1. $J(\alpha _{2}\alpha _{1}^{l})<J(\alpha _{1}\alpha _{2}^{l}).$\\ 
Case 2. $J(\alpha _{1}\alpha _{2}^{l})<J(\alpha _{2}\alpha _{1}^{l}).$  

 We may assume that we have Case 1 (when we have Case 2, we can show 
 all statements of our theorem, using the same method as below).  
Let $A:= K(\alpha _{1}\alpha _{2}^{l})\setminus $ 
int$(K(\alpha _{2}\alpha _{1}^{l})).$ 
By Claim 1, we have that 
$J(\alpha _{1}\alpha _{2}^{l})$ and 
$J(\alpha _{2}\alpha _{1}^{l})$ are quasicircles. 
Moreover, 
since $H_{0}\in {\cal G}_{dis }$ and 
$H_{0}$ is hyperbolic, 
we must have 
$P^{\ast }(H_{0})\subset $
int$(K(\alpha _{2}\alpha _{1}^{l})).$ 
Therefore, 
it follows that if we take a small open neighborhood $U$ of 
$A$, then there exists a number $n\in \NN $ such that, 
setting $h_{1}:= (\alpha _{2}\alpha _{1}^{l} )^{n}$ 
and $h_{2}:= (\alpha _{1}\alpha _{2}^{l})^{n}$, 
we have that 
\begin{equation}
\label{cantorqcpfeq01}
h_{1}^{-1}(\overline{U})\cup 
h_{2}^{-1}(\overline{U})\subset U \mbox{ and } 
h_{1}^{-1}(\overline{U})\cap h_{2}^{-1}(\overline{U})=\emptyset . 
\end{equation}
%We set $H:= \langle g_{1},g_{2}\rangle .$ 
%Then, since $H$ is a subsemigroup of $H_{0}$ and 
%$H_{0}$ is hyperbolic, 
%we have that $H$ is hyperbolic. 
%Moreover, 
%since $J(g_{1})=J(\alpha _{2}\alpha _{1}^{l}) $ and 
%$J(g_{1})\subset J(H)$, 
%we have $J(H)\cap V\neq \emptyset .$ 
Moreover, combining 
\cite[Lemma 3.4(2)]{SdpbpIII}
%Lemma~\ref{fibfundlem}-\ref{fibfundlem2} 
and that $J(h_{1})\cap V\neq \emptyset $, 
we get that there exists a $k\in \NN $ such that 
$J(h_{2}h_{1}^{k})\cap V\neq \emptyset .$ 
We set $g_{1}:= h_{1}^{k+1}$ and $g_{2}:= h_{2}h_{1}^{k}.$ 
Moreover, we set $H:= \langle g_{1},g_{2}\rangle .$ Since 
$H$ is a subsemigroup of $H_{0}$ and $H_{0}$ is hyperbolic, 
we have that $H$ is hyperbolic. Moreover, 
(\ref{cantorqcpfeq01}) implies that 
$g_{1}^{-1}(\overline{U})\cup 
g_{2}^{-1}(\overline{U})\subset U$ and 
$g_{1}^{-1}(\overline{U})\cap g_{2}^{-1}(\overline{U})=\emptyset .$ 
Hence, we have shown that for the semigroup $H=\langle g_{1},g_{2}\rangle $, 
statements \ref{cantorqc1},\ref{cantorqc2},
and \ref{cantorqc3} of Theorem~\ref{cantorqc} hold. 

 From statement \ref{cantorqc2} and 
\cite[Corollary 3.2]{HM1},  
%  Lemma~\ref{hmslem}-\ref{backmin}, 
%   and Lemma~\ref{hmslem}-\ref{bss}, 
 we obtain $J(H)\subset \overline{U}$ 
 and $g_{1}^{-1}(J(H))\cap g_{2}^{-1}(J(H))=\emptyset .$  
Combining this with 
%Lemma~\ref{hmslem}-\ref{bss} and 
%Lemma~\ref{fiblem}-\ref{fibfundlem5},  
\cite[Lemma 2.4]{S3} and \cite[Lemma 3.5(2)]{SdpbpIII}, 
it follows that the 
skew product $f:\GN _{1}\times \CCI 
\rightarrow \GN _{1}\times \CCI $ 
associated with $\G _{1}= \{ g_{1},g_{2}\} $ 
satisfies that 
$J(H)$ is equal to the disjoint union of 
the sets $\{ \hat{J}_{\g }(f)\} _{\g \in \GN _{1}}.$ 
Moreover, since $H$ is hyperbolic, 
\cite[Theorem 2.14-(2)]{S1} implies that for each 
$\g \in \G _{1}^{\NN }$, 
$\hat{J}_{\g }(f)=J_{\g }(f).$  
In particular, 
the map $\g \mapsto J_{\g }(f)$ 
from $\GN _{1}$ into the space of 
non-empty compact sets in $\CCI $,  is injective. 
%Moreover, 
Since $J_{\g }(f)$ is connected for each 
$\g \in \GN _{1}$ (Claim 1), 
%Lemma~\ref{fibconnlem}), 
it follows that for each connected component $J$ of $J(H)$,  
there exists an element $\g \in \GN _{1}$ such that 
$J=J_{\g }(f).$ 
Furthermore, by Claim 1,  
each connected component $J$ of $J(H)$ is a $K$-quasicircle, 
where $K$ is a constant not depending on $J.$ 
 Moreover, by \cite[Theorem 2.14-(4)]{S1}, the map 
 $\g \mapsto J_{\g }(f)$ from $\GN _{1}$ 
 into the space of non-empty compact sets in $\CCI $, 
 is continuous with respect to 
 the Hausdorff metric. 
% in the space of non-empty compact subsets of 
% $\CCI .$ 
Moreover, by Theorem~\ref{mainth1}, 
$(\{ J_{\gamma }(f)\} _{\g \in \GN }, \leq )$ is totally ordered. 
Therefore, we have shown that statements 
4(a), 4(b), 4(c), and 4(d) 
% \ref{cantorqc4a},\ref{cantorqc4b},\ref{cantorqc4c}, and \ref{cantorqc5} 
hold for $H=\langle g_{1},g_{2}\rangle $ and 
 $f:\GN _{1}\times \CCI \rightarrow \GN _{1}\times \CCI .$  

 We now show that statement 4(e) holds. 
%\ref{cantorqc6} holds.
Since we are assuming Case 1, 
Proposition~\ref{bminprop} implies that 
$\{ h_{1},h_{2}\} _{\min }=\{ h_{1}\} .$ 
Hence $J(g_{1})<J(g_{2}).$ 
Combining this with Proposition~\ref{bminprop} and statement 4(b), 
%\ref{cantorqc4b},   
we obtain  
\begin{equation}
\label{cantorqc6pfeq1}
J(g_{1})= J_{\min }(H) \mbox{ and }
J(g_{2})= J_{\max }(H).
\end{equation} 
Moreover, since 
$J(g_{1})=J(\alpha _{2}\alpha _{1}^{l})$, 
$J(\alpha _{2}\alpha _{1}^{l})\cap V\neq \emptyset $, 
$J(g_{2})=J(h_{2}h_{1}^{k})$, and $J(h_{2}h_{1}^{k})\cap 
V\neq \emptyset $, 
it follows that 
\begin{equation}
\label{cantorqc6pfeq2}
J_{\min }(H)\cap V\neq \emptyset \mbox{ and }
J_{\max }(H)\cap V\neq \emptyset .
\end{equation}
Let $\g \in \GN $ be an element such that 
$J_{\g }(f)\cap (J_{\min }(H)\cup J_{\max }(H))=\emptyset .$ 
%Combining statement \ref{cantorqc5} and 
%Lemma~\ref{fiborder}, 
By statement 4(b), 
%\ref{cantorqc4b}, 
we obtain 
\begin{equation}
\label{cantorqc6pfeq3}
J_{\min }(H)<J_{\g }(f)<J_{\max }(H).
\end{equation}
Since we are assuming $V$ is connected, 
combining (\ref{cantorqc6pfeq2}) and 
(\ref{cantorqc6pfeq3}), 
we obtain $J_{\g }(f)\cap V\neq \emptyset .$ 
Therefore, we have proved that 
statement 4(e) holds. 
%\ref{cantorqc6} holds. 

 We now show that statement 4(f) holds. 
%\ref{cantorqc7} holds. 
 To show that, 
% let $t\in \NN $ be a large number and 
let $\omega  =(\omega  _{1},\omega  _{2},\ldots )\in \GN _{1}$ be an 
element such that 
%$\g _{1}=\g _{2}=\cdots =\g _{2t}=1$, 
$\sharp (\{ j\in \NN 
\mid \omega _{j}=g_{1}\} )=
\sharp (\{ j\in \NN \mid \omega _{j}=g_{2}\} )=\infty .$ 
%$\g _{2t+2j+1}=g_{1} $ for each $j\in \NN \cup \{ 0\} $, and 
%$\g _{2t+2j}=g_{2}$ for each $j\in \NN .$ 
%Taking $t\in \NN$ large enough, 
%Lemma~\ref{fibfundlem}-\ref{fibfundlem2} tells us that 
%we may assume that 
%$J_{\g }(f)\cap V\neq \emptyset .$ 
For each $r\in \NN $, 
let $\omega  ^{r}=(\omega  ^{r}_{1},\omega  ^{r}_{2},\ldots )\in \GN _{1}$ 
be the element such that 
$\begin{cases}
\omega  ^{r}_{j}=\omega  _{j} \ \ (1\leq j\leq r), \\  
\omega  ^{r}_{j}=g_{1}     \ \ (j\geq r+1).
\end{cases} $
 Moreover, 
let $\rho ^{r}=(\rho ^{r}_{1},\rho ^{r}_{2},\ldots )\in \GN _{1}$ 
be the element such that 
 $\begin{cases}
\rho^{r}_{j}=\omega  _{j} \ \ (1\leq j\leq r), \\  
\rho ^{r}_{j}=g_{2}     \ \ (j\geq r+1).
\end{cases} $
%By Proposition~\ref{bminprop} and 
%$J(g_{1})<J(g_{2})$, we have 
%$J(g_{1})\subset 
%J_{\min }(H)$ and 
%$J(g_{2})\subset J_{\max }(H).$ 
Combining (\ref{cantorqc6pfeq1}) and  
statements 4(a) and 4(b),
% \ref{cantorqc4a}, and statement \ref{cantorqc4b},   
%which we have shown 
%in the previous paragraph, 
we see that for each $r\in \NN $, 
$J(g_{1})<J_{\sigma ^{r}(\omega  )}(f)<J(g_{2}).$ 
Hence, by statement \ref{mainth1-3} of Theorem~\ref{mainth1},  
we get that for each $r\in \NN $, 
$(f_{\omega  ,r})^{-1}(J(g_{1}))<
(f_{\omega  ,r})^{-1}\left(J_{\sigma ^{r}(\omega  )}(f)\right)
<(f_{\omega  ,r})^{-1}(J(g_{2}))$. 
% where 
%$f_{\omega  ,r}(y)=\pi _{\CCI }(f^{r}(\omega ,y)).$   
Since we have 
$ (f_{\omega  ,r})^{-1}(J(g_{1}))=J_{\omega  ^{r}}(f)$, 
$(f_{\omega  ,r})^{-1}\left(J_{\sigma ^{r}(\omega  )}(f)\right)
=J_{\omega  }(f)$, 
and $(f_{\omega  ,r})^{-1}(J(g_{2}))$ $=J_{\rho ^{r}}(f)$, 
it follows that 
\begin{equation}
\label{cantorqcpfeq1}
J_{\omega  ^{r}}(f)<J_{\omega  }(f)<
J_{\rho ^{r}}(f),  
\end{equation}
for each $r\in \NN .$ 
Moreover, since 
$\omega  ^{r}\rightarrow \omega  $ and 
$\rho ^{r}\rightarrow \omega  $ in 
$\GN _{1}$ as $r\rightarrow \infty $, 
statement 4(d) 
%\ref{cantorqc5}
%, which we have shown in the previous paragraph, 
implies that 
$J_{\omega  ^{r}}(f)\rightarrow 
J_{\omega  }(f)$ and 
$J_{\rho ^{r}}(f)\rightarrow 
J_{\omega  }(f)$ as $r\rightarrow \infty $, 
with respect to the Hausdorff metric. 
Combining these with (\ref{cantorqcpfeq1}) and statements 4(b) and 4(c), 
%\ref{cantorqc4b} 
%and statement \ref{cantorqc4c}, 
we get that 
for any connected component $W$ of $F(H)$, 
we must have $\partial W\cap J_{\omega  }(f)
=\emptyset .$ Since $F(G)\subset F(H)$, 
it follows that 
for any connected component $W'$ of $F(G)$, 
$\partial W'\cap J_{\omega  }(f)=\emptyset .$
Therefore, we have shown that statement 4(f) holds. 
%\ref{cantorqc7} holds. 
%Thus, we have shown that all statements in our theorem hold. 
%provided that we have the case 1.
%
% Suppose that we have the case 2: $J(\alpha _{1}\alpha _{2}^{l})<J(\alpha _{2}\%alpha _{1}^{l}).$  
%
%Then, let $A:= K(\alpha _{2}\alpha _{1}^{l})\setminus $ 
%int$(K(\alpha _{1}\alpha _{2}^{l})).$
%Using the same method as that in case 1, 
%we obtain that if we take a small 
%open neighborhood $U$ of $A$, there exists 
%a number $n\in \NN $ such that, 
%setting $g_{1}:= (\alpha _{2}\alpha _{1}^{l} )^{n}$ 
%and $g_{2}:= (\alpha _{1}\alpha _{2}^{l})^{n}$, 
%we have that $g_{1}^{-1}(\overline{U})\cup 
%g_{2}^{-1}(\overline{U})\subset U$ and 
%$g_{1}^{-1}(\overline{U})\cap g_{2}^{-1}(\overline{U})=\emptyset .$ 
%Continuing the same argument as that in case 1, 
%we obtain that 
%all statements \ref{cantorqc1},$\ldots $,\ref{cantorqc6} 
%hold for the semigroup 
%$H=\langle g_{1},g_{2}\rangle  $ and 
%the skew product $f:\{ g_{1},g_{2}\} ^{\NN }\times 
%\CCI \rightarrow \{ g_{1},g_{2}\} ^{\NN }\times \CCI .$  
 
  Thus, we have proved Theorem~\ref{cantorqc}.
\qed 
\subsection{Proofs of the results in \ref{fjjq}}
\label{Proofs of fjjq}
In this section, we demonstrate Theorems~\ref{mainthjbnq} and \ref{t:tghyp}.   
We need the following notations and lemmas.
\begin{df}
Let $h$ be a polynomial with $\deg (h)\geq 2.$ Suppose that 
$J(h)$ is connected. Let $\psi $ be a biholomorphic map 
$\CCI \setminus \overline{D(0,1)}\rightarrow  
F_{\infty }(h)$ with $\psi (\infty )=\infty $ such that 
$\psi ^{-1}\circ h\circ \psi (z)=z^{\deg (h)}$, 
for each $z\in \CCI \setminus \overline{D(0,1)}.$ 
(For the existence of the biholomorphic map $\psi $, see 
\cite[Theorem 9.5]{M}.) For each 
$\theta \in \partial D(0,1)$, we set 
$T(\theta ):= \psi (\{ r\theta \mid 1<r\leq \infty \} ).$ 
This is called the external ray (for $K(h)$) 
with angle $\theta .$ 
\end{df}

\begin{lem}
\label{jbnqlem1}
Let $h$ be a polynomial with $\deg (h)\geq 2.$ 
Suppose that $J(h)$ is connected and locally connected 
and $J(h)$ is not a Jordan curve. Moreover, suppose that 
there exists an attracting periodic point of $h$ in $K(h).$ 
%int $K(h)\neq \emptyset $ and  
%int $K(h)$ is a union of basins of attractions  
%for attracting periodic points of $h.$   
Then, for any $\epsilon >0$, there exist a point $p\in J(h)$  
and elements $\theta _{1}, \theta _{2}\in \partial D(0,1)$ 
with $\theta _{1}\neq \theta _{2}$, such that all of the following hold.
\begin{enumerate}
\item 
For each $i=1,2$,  the external ray $T(\theta _{i})$ lands 
at the point $p.$ 
\item Let $V_{1}$ and $V_{2}$ be the two connected components of  
 $\CCI \setminus (T(\theta _{1})\cup T(\theta _{2})\cup \{ p\} ).$  
 Then, for each $i=1,2,$ 
 $V_{i}\cap J(h)\neq \emptyset .$ Moreover, there exists an $i$ such that 
 diam $(V_{i}\cap K(h))\leq \epsilon .$ 
\end{enumerate}
\end{lem} 
\begin{proof}
Let $\psi : \CCI \setminus \overline{D(0,1)}\rightarrow 
F_{\infty }(h)$ be a biholomorphic map with $\psi (\infty )=\infty $ 
such that for each $z\in 
\CCI \setminus \partial D(0,1),  
\psi ^{-1}\circ h\circ \psi (z)=z^{\deg (h)}.$ 
Since $J(h)$ is locally connected, the map 
$\psi : \CCI \setminus \overline{D(0,1)}\rightarrow 
F_{\infty }(h)$ extends continuously over $\partial D(0,1)$, 
mapping $\partial D(0,1)$ onto $J(h).$ 
Moreover, since $J(h)$ is not a Jordan curve, it follows that 
there exist a point $p_{0}\in J(h)$ and two points $t_{1},t_{2}\in 
\partial D(0,1)$ with $t_{1}\neq t_{2}$ such that 
two external rays 
$T(t _{1})$ and $T(t _{2})$ land at the same 
point $p_{0}.$ 
Considering a higher iterate of $h$ if necessary, 
we may assume that there exists an attracting fixed point of $h$ in 
int$(K(h)).$ 
%consists of a union of 
%basins of attraction for attracting fixed points of $h.$ 
Let $a\in $ int$(K(h))$ be an attracting fixed point of $h$ 
and let $U$ be the connected component of int$(K(h))$ containing $a.$ 
Then, there exists a critical point $c\in U$ of $h.$ 
Let $V_{0}$ be the connected component of 
$\CCI \setminus (T(t_{1})\cup T(t_{2})\cup \{ p_{0}\} )$ 
containing $a.$ 
Moreover, for each $n\in \NN $, 
let $V_{n}$ be the connected component of $(h^{n})^{-1}(V_{0})$ 
containing $a.$ 
Since $c\in U$, we get that for each $n\in \NN $, 
$c\in V_{n}.$ Hence, setting 
$e_{n}:= \deg (h^{n}:V_{n}\rightarrow V_{0})$, 
it follows that 
%\begin{equation}
%\label{jbnqlem1eq1}
$e_{n}\rightarrow \infty \mbox{ as }n\rightarrow \infty .$ 
%\end{equation} 
We fix an $n\in \NN $ satisfying $e_{n}>d$, where 
$d:= \deg (h).$ 
Since 
$\deg (h^{n}:V_{n}\cap F_{\infty }(h)\rightarrow 
V_{0}\cap F_{\infty }(h))=\deg (h^{n}:V_{n}\rightarrow V_{0})$, 
we have that the number of connected components of 
$V_{n}\cap F_{\infty }(h)$ is equal to $e_{n}.$ 
Moreover, every connected component of $V_{n}\cap F_{\infty }(h)$ 
is a connected component of 
$(h^{n})^{-1}(V_{0}\cap F_{\infty }(h)).$ Hence, it follows that 
there exist mutually disjoint arcs $\xi _{1},\xi _{2},\ldots ,\xi _{e_{n}}$ 
in $\CC $ satisfying all of the following.
\begin{enumerate}
\item For each $j$,   
$h^{n}(\xi _{j})=(T(t_{1})\cup T(t_{2})\cup \{ p_{0}\} )\cap \CC .$ 
\item 
For each $j$, 
$\xi _{j}\cup \{ \infty \} $ is the closure of union of 
two external rays and $\xi _{j}\cup \{ \infty \} $ is a Jordan curve.
\item We have $\partial V_{n}=\xi _{1}\cup \cdots \cup \xi _{e_{n}}\cup \{ \infty \} .$ 
\end{enumerate} 
For each $j=1,\ldots ,e_{n}$, 
let $W_{j}$ be the connected component 
of $\CCI \setminus (\xi _{j}\cup \{ \infty \} )$ that does not contain 
$V_{n}.$ 
Then, each $W_{j}$ is a connected component of 
$\CCI \setminus \overline{V_{n}}.$ Hence, 
for each $(i,j)$ with $i\neq j$, 
$W_{i}\cap W_{j}=\emptyset .$ 
Since the number of critical values of $h$ in $\CC $ is less than or equal to 
$d-1$, we have that 
$\sharp (\{1\leq j\leq e_{n}\mid W_{j}\cap CV(h)=\emptyset \} )\geq 
e_{n}-(d-1).$ 
%where $CV(h)$ denotes the set of all critical values of $h.$ 
Therefore, denoting by $u_{1,j}$ 
the number of well-defined inverse branches of 
$h$ on $W_{j}$, we obtain
$\sum _{j=1}^{e_{n}}u_{1,j}\geq d(e_{n}-(d-1))\geq d.$ 
Inductively, denoting by $u_{k,j}$ the number of well-defined inverse 
branches of $h^{k}$ on $W_{j}$, we obtain
\begin{equation}
\label{jbnqlem1pfeqa1}
\sum _{j=1}^{e_{n}}u_{k,j}\geq d(d-(d-1))\geq d, \mbox{for each }k\in \NN .
\end{equation}
For each $k\in \NN $, we take a well-defined   
inverse branch $\zeta _{k}$ of $h^{k}$ on a 
domain $W_{j}$, and let $B_{k}:=\zeta _{k}(W_{j}).$ 
Then, $h^{k}:B_{k}\rightarrow W_{j}$ is biholomorphic.
Since $\partial B_{k}$ is the closure of finite union of 
external rays and $h^{n+k}$ maps each connected component of 
$(\partial B_{k})\cap \CC $ 
%is a connected component of 
onto $(T(t_{1})\cup T(t_{2})\cup \{ p_{0}\} )\cap \CC $, 
%$(h^{k})^{-1}((T(t_{1}\cup T(t_{2})\cup \{ p_{0}\} )\cap \CC )$
$B_{k}$ is a Jordan domain. 
Hence, $h^{k}:B_{k}\rightarrow W_{j}$ induces a 
homeomorphism $\partial B_{k}\cong \partial W_{j}.$ 
Therefore, $\partial B_{k}$ is the closure of 
union of two external rays, which implies that 
$B_{k}\cap F_{\infty }(h)$ is a connected component of 
$(h^{k})^{-1}(W_{j}\cap F_{\infty }(h)).$ 
Hence, we obtain  
\begin{equation}
\label{jbnqlem1eq2}
l\left( \overline{\psi ^{-1}(B_{k}\cap F_{\infty }(h))}\cap \partial D(0,1)
\right) \rightarrow 0 \mbox{ as } k\rightarrow \infty ,
\end{equation}  
where $l(\cdot )$ denotes the arc length of a subarc of 
$\partial D(0,1).$ 
Since $\psi :\CCI \setminus \overline{D(0,1)}\rightarrow  F_{\infty }(h)$ 
extends continuously over $\partial D(0,1)$, 
(\ref{jbnqlem1eq2}) implies that 
diam $(B_{k}\cap J(h))\rightarrow 0$ as $k\rightarrow \infty .$ 
Hence, there exists a $k\in \NN $ such that 
diam $(B_{k}\cap K(h))\leq \epsilon .$ 
Let $\theta _{1},\theta _{2}\in \partial D(0,1)$ be 
two elements such that 
$\partial B_{k}=\overline{T(\theta _{1})\cup T(\theta _{2})}.$ 
Then, there exists a point $p\in J(h)$ such that 
each $T(\theta _{i})$ lands at the point $p.$ 
By \cite[Lemma 17.5]{M}, any of two connected components of 
$\CCI \setminus (T(\theta _{1})\cup T(\theta _{2})\cup \{ p\} )$ 
intersects $J(h).$ 

 Thus, we have proved Lemma~\ref{jbnqlem1}.
 \end{proof}
\begin{lem}
\label{060414astlem1}
Let $G$ be a polynomial semigroup generated by a 
compact subset $\G $ of {\em Poly}$_{\deg \geq 2}.$
Let $f:\GNCR $ be the skew product associated with the family 
$\G .$ Suppose $G\in {\cal G}_{dis}.$ 
Let $m\in \NN $ and suppose that there exists an element 
$(h_{1},\ldots ,h_{m})\in \G ^{m}$ such that setting 
$h=h_{m}\circ \cdots \circ h_{1}$, 
$J(h)$ is connected and locally connected, and 
$J(h)$ is not a Jordan curve. Moreover, suppose that 
there exists an attracting periodic point of $h$ in 
$K(h).$ Let $\alpha =(\alpha _{1},\alpha _{2},\ldots )\in 
\GN $ be the element such that for each 
$k,l\in \NN \cup \{ 0\} $ with $1\leq l\leq m$, 
$\alpha _{km+l}=h_{l}.$ 
Let $\rho _{0}\in \G \setminus \G _{\min }$ be an element and 
let 
$\beta =(\rho _{0},\alpha _{1},\alpha _{2},\ldots )\in \GN .$ 
Moreover, let 
$\psi _{\beta }:\CCI \setminus \overline{D(0,1)}
\rightarrow A_{\beta }(f)$ be a biholomorphic map 
with $\psi _{\beta }(\infty )=\infty .$ 
Furthermore, for each $\theta \in 
\partial D(0,1)$, 
let 
$T_{\beta }(\theta )=\psi _{\beta }(\{ r\theta \mid 
1<r\leq \infty \} ).$ 
Then, for any $\epsilon >0$, there exist a point 
$p\in J_{\beta }(f)$ and elements 
$\theta _{1},\theta _{2}\in \partial D(0,1)$ with 
$\theta _{1}\neq \theta _{2}$, such that 
both of the following statements 1 and 2 hold.
\begin{enumerate}
\item For each $i=1,2$, $T_{\beta }(\theta _{i})$ lands at $p.$
\item Let $V_{1}$ and $V_{2}$ be the two connected components 
of 
$\CCI \setminus (T_{\beta }(\theta _{1})\cup T_{\beta }(\theta _{2})
\cup \{ p\} ).$ Then, 
for each $i=1,2$, $V_{i}\cap J_{\beta }(f)\neq \emptyset .$ 
Moreover, there exists an $i$ such that 
diam $(V_{i}\cap K_{\beta }(f))\leq \epsilon $ and 
such that $V_{i}\cap J_{\beta }(f)\subset 
\rho _{0}^{-1}(J(G))\subset \CC \setminus P(G).$ 
\end{enumerate}  
\end{lem}
\begin{proof}
We use the notation and argument in the proof of 
Lemma~\ref{jbnqlem1}. 
Taking a higher iterate of $h$, 
we may assume that 
$d:=\deg (h)>\deg (\rho _{0}).$ 
Then, from (\ref{jbnqlem1pfeqa1}), 
it follows that 
for each $k\in \NN $, 
we can take a well-defined inverse branch 
$\zeta _{k}$ of 
$h^{k}$ on a domain $W_{j}$ such that 
setting $B_{k}:=\zeta _{k}(W_{j})$, 
$B_{k}$ does not contain any critical value of 
$\rho _{0}.$ 
By (\ref{jbnqlem1eq2}), 
there exists a $k\in \NN $ such that 
diam $(B_{k}\cap J(h))\leq \epsilon '$, where 
$\epsilon '>0$ is a small number. 
Let $B$ be a connected component of 
$\rho _{0}^{-1}(B_{k}).$ 
Then, there exist a point 
$p\in J_{\beta }(f)$ and elements 
$\theta _{1},\theta _{2}\in \partial D(0,1)$ with 
$\theta _{1}\neq \theta _{2}$ such that 
for each $i=1,2,$ $T_{\beta }(\theta _{i})$ lands at 
$p$, and such that $B$ is a connected component of 
$\CCI \setminus (T_{\beta }(\theta _{1})\cup 
T_{\beta }(\theta _{2})\cup \{ p\} ).$ Taking $\epsilon '$ so small, 
we obtain diam $(B\cap K_{\beta }(f))=$ diam $(B\cap J_{\beta }(f))
\leq \epsilon .$ Moreover, 
since $\rho _{0}\in \G \setminus \G _{\min }$, 
%combining Theorem~\ref{mainth2}-\ref{mainth2-3} and 
%Theorem~\ref{mainth2}-\ref{mainth2-4-2}, 
by \cite[Theorem 2.20.5(b), Theorem 2.20.2]{SdpbpI} 
we obtain 
$J_{\beta }(f)
=\rho _{0}^{-1}(J(h))
\subset \rho _{0}^{-1}(J(G))
\subset  
\CC \setminus P(G).$ 
Hence, $B\cap J_{\beta }(f)\subset 
\rho _{0}^{-1}(J(G))\subset 
\CC \setminus P(G).$ 
Therefore, we have proved Lemma~\ref{060414astlem1}.  
\end{proof}
 We now demonstrate statement \ref{mainthjbnq1} of Theorem~\ref{mainthjbnq}.\\ 
{\bf Proof of statement \ref{mainthjbnq1} of Theorem~\ref{mainthjbnq}:}  
Let $\g $ be as in statement \ref{mainthjbnq1} of Theorem~\ref{mainthjbnq}.  
Then, by Theorem~\ref{mainth3-2},
%\begin{equation}
%\label{mainthjbnq1pfeq1} 
$J_{\g }(f)$ is a Jordan curve. 
%\end{equation}
Moreover, setting $h=h_{m}\circ \cdots \circ h_{1}$, 
since $h$ is hyperbolic and $J(h)$ is not a quasicircle, 
%\begin{equation}
%\label{mainthjbnq1pfeq2}
$J(h)$ is not a Jordan curve.
%\end{equation}
Combining this with 
%(\ref{mainthjbnq1pfeq1}), (\ref{mainthjbnq1pfeq2}), 
%Lemma~\ref{060416lemast} 
\cite[Lemma 4.5]{SdpbpIII}
and 
Lemma~\ref{jbnqlem1}, it follows that 
$J_{\g }(f)$ is not a quasicircle. 
Moreover, $A_{\g }(f)$ is a John domain 
(cf. \cite[Theorem 1.12]{S4}). 
Combining the above arguments with \cite[Theorem 9.3]{NV}, 
we conclude that the bounded component $U_{\g }$ of 
$F_{\g }(f)$ is not a John domain. 

 Thus, we have proved statement \ref{mainthjbnq1} of Theorem~\ref{mainthjbnq}. 
\qed 

\ 

 We now demonstrate statement \ref{mainthjbnq2} of Theorem~\ref{mainthjbnq}.\\ 
{\bf Proof of statement \ref{mainthjbnq2} of Theorem~\ref{mainthjbnq}:}
Let $\rho _{0}, \beta ,\g $ be as in statement \ref{mainthjbnq2} of 
Theorem~\ref{mainthjbnq}. 
By Theorem~\ref{mainth3-2}, 
$J_{\g }(f)$ is a Jordan curve. 
By statement \ref{mainth2-4} of Theorem~\ref{mainth2},  
we have 
$\emptyset \neq $ int$(\hat{K}(G))\subset $ int$(K(h)).$ 
Moreover, $h$ is semi-hyperbolic. Hence, 
$h$ has an attracting periodic point in $K(h).$ 
Combining 
%Lemma~\ref{060416lemast} 
\cite[Lemma 4.5]{SdpbpIII}
and 
Lemma~\ref{060414astlem1}, we get that 
$J_{\g }(f)$ is not a quasicircle. 
Combining this with the argument in the proof of 
statement \ref{mainthjbnq1} of 
Theorem~\ref{mainthjbnq},  
it follows that 
$A_{\g }(f)$ is a John domain, but the bounded component 
$U_{\g }$ of $F_{\g }(f)$ is not a John domain. 

 Thus, we have proved statement \ref{mainthjbnq2} of Theorem~\ref{mainthjbnq}. 
\qed 

We now prove Theorem~\ref{t:tghyp}.\\ 
{\bf Proof of Theorem~\ref{t:tghyp}}: 
By \cite[Theorem 3.17]{S7}, we have 
$h_{1}^{-1}(J(G))\cap h_{2}^{-1}(J(G))=\emptyset .$ 
Combining this with \cite[Lemma 3.5(2), Lemma 3.6]{SdpbpIII} and \cite[Theorem 2.14(2)]{S1}, 
we obtain that $J(G)=\amalg _{\g \in \GN} J_{\g }(f)$ (disjoint union) and 
for each connected component $J$ of $J(G)$, there exists a unique $\g \in \GN $ such that 
$J=J_{\g }(f).$ 

In order to prove that we have exactly one of statements 1 and 2, 
suppose that $J(h_{1})$ and $J(h_{2})$ are Jordan curves. 
By Proposition~\ref{bminprop}, we may assume that 
$J_{\min }(G)=J(h_{1})$ and $J_{\max }(G)=J(h_{2}).$ 
Then by \cite[Theorem 2.20.5(b)]{SdpbpI}, we have 
$\mbox{int}(K(h_{1}))=\mbox{int}(\hat{K}(G)).$ Thus $P^{\ast }(G)$ is included in a connected component 
of int$(\hat{K}(G)).$ Combining this with \cite[Proposition 2.25]{SdpbpIII}, 
we obtain that statement 1 of Theorem~\ref{t:tghyp} holds.  

We now suppose that $J(h_{j})$ is not a Jordan curve. Then, as above, we have 
$J_{\min }(G)=J(h_{j})$ and $J_{\max }(G)=J(h_{i})$, where $i\neq j.$ 
By \cite[Theorem 2.20.4]{SdpbpI}, $J(h_{i})$ is a quasicircle.  
Moreover, combining statement \ref{mainth3-3} of Theorem~\ref{mainth3}, 
statement~\ref{mainthjbnq1} of Theorem~\ref{mainthjbnq}
%-\ref{mainth3-3}, Theorem~\ref{mainthjbnq}-\ref{mainthjbnq1} and 
and \cite[Lemma 4.4]{SdpbpIII}, we obtain that exactly one of statements (a),(b),(c) of  
statement 2 of Theorem~\ref{t:tghyp} holds. Thus, we have proved Theorem~\ref{t:tghyp}.   
\qed 
\subsection{Proofs of the results in \ref{random}}
\label{Proofs of random}
In this subsection, we will demonstrate results in 
Section~\ref{random}. 

 we now prove Corollary~\ref{rancor1}.\\ 
{\bf Proof of Corollary~\ref{rancor1}:} 
By Remark~\ref{jminrem}, there exists 
a compact subset 
$S$ of $\G \setminus \G _{\min }$ such that 
the interior of $S$ with respect to the 
space $\G $ is not empty.  
%Under the notation in Theorem~\ref{mainth3}, 
Let ${\cal U}:=R(\G ,S).$ Then, 
it is easy to see that 
${\cal U}$ is residual in $\GN $, and that 
for each Borel probability measure 
$\tau $ on Poly$_{\deg \geq 2}$ with $\G _{\tau }=\G $, 
we have $\tilde{\tau }({\cal U})=1.$ Moreover, 
by statements \ref{mainth3-0} and \ref{mainth3-1} of Theorem~\ref{mainth3}, 
%-\ref{mainth3-0} and 
%Theorem~\ref{mainth3}-\ref{mainth3-1}, 
each $\g \in {\cal U}$ satisfies properties 
\ref{rancor1-1},\ref{rancor1-2},\ref{rancor1-3}, and 
\ref{rancor1-4} of Corollary~\ref{rancor1}. Hence, 
we have proved Corollary~\ref{rancor1}.
\qed  

We now prove Corollary~\ref{c:rancor2sh}. \\ 
{\bf Proof of Corollary~\ref{c:rancor2sh}:}   
%Setting ${\cal U}:= \{ \g \in R(\G , \G \setminus \G _{\min }) \mid \exists \{ n_{k}\} \mbox{ s.t. }\sigma ^{n_{k}}(\g )\rightarrow \alpha \} $, 
Let ${\cal U}$ be as in the proof of Corollary~\ref{rancor1}. 
By using Theorem~\ref{mainth3-2}, 
it is easy to see that statement~\ref{c:rancor2sh1} holds. 
We now prove statement~\ref{c:rancor2sh2}. 
From our assumption, there exist $h_{1},\ldots ,h_{m}\in \G $ 
such that $J(h_{m}\circ \cdots \circ h_{1})$ is not a quasicircle. 
Let $\alpha =(\alpha _{1},\alpha _{2},\ldots )\in \G ^{\NN }$ be such that 
for each $k,l\in \NN \cup \{ 0\} $ with $1\leq l\leq m$, 
$\alpha _{km+l}=h_{l}.$ Let $\rho _{0}\in \G \setminus \G _{\min }$ be an element and 
let $\beta =(\rho _{0}, \alpha _{1},\alpha _{2},\ldots )\in \G ^{\NN }.$ 
Let 
${\cal V}:= \{ \g \in R(\G, \G \setminus \G _{\min }) \mid \exists \{ n_{k}\} \mbox{ s.t. } \sigma ^{n_{k}}(\g )\rightarrow \beta \} .$ 
By statement~\ref{mainthjbnq2} of Theorem~\ref{mainthjbnq},  
 ${\cal V} $ satisfies the desired properties. 
%Thus we have proved Corollary~\ref{c:rancor2sh}.   
\qed

\end{document}